\documentclass[11pt]{amsart}
\usepackage{tikz}
\usepackage{graphicx}
\usepackage{amsmath}
\usepackage{amssymb}
\usepackage{amsfonts}
\usepackage{verbatim}
\usepackage{scalefnt}
\usepackage{multirow}
\usepackage{enumerate}
\usepackage{mathtools}
\usepackage{float}
\usepackage{enumitem}
\restylefloat{figure}
\usepackage[margin=3.4cm]{geometry}

\usetikzlibrary{decorations.pathreplacing}

\usepackage{tikz}
\usetikzlibrary{patterns}
\usepackage{fancyhdr}
\usepackage{hyperref}

\begin{document}

\def\COMMENT#1{}
\let\COMMENT=\footnote

\newcommand{\case}[1]{\medskip\noindent{\bf Case #1} }
\newcommand{\step}[1]{\medskip\noindent{\bf Step #1} }

\newtheorem{theorem}{Theorem}
\newtheorem{problem}[theorem]{Problem}
\newtheorem{lemma}[theorem]{Lemma}
\newtheorem{proposition}[theorem]{Proposition}
\newtheorem{corollary}[theorem]{Corollary}
\newtheorem{conjecture}[theorem]{Conjecture}
\newtheorem{claim}[theorem]{Claim}
\newtheorem{definition}[theorem]{Definition}
\newtheorem*{definition*}{Definition}
\newtheorem{fact}[theorem]{Fact}
\newtheorem{observation}[theorem]{Observation}
\newtheorem{question}[theorem]{Question}
\newtheorem{remark}[theorem]{Remark}

\numberwithin{equation}{section}
\numberwithin{theorem}{section}

\def\eps{{\varepsilon}}
\renewcommand{\epsilon}{\varepsilon}
\newcommand{\cP}{\mathcal{P}}
\newcommand{\cG}{\mathcal{G}}
\newcommand{\cT}{\mathcal{T}}
\newcommand{\cL}{\mathcal{L}}
\newcommand\ex{\ensuremath{\mathrm{ex}}}
\newcommand{\eul}{e}
\newcommand{\pr}{\mathbb{P}}
\newcommand{\expect}{\mathbb{E}}
\newcommand{\phiind}{\phi^{\rm ind}}

\title[On the rational Tur\'an exponents conjecture]{On the rational Tur\'an exponents conjecture}

\author{Dong Yeap Kang, Jaehoon Kim, and Hong Liu}
\address[Dong Yeap Kang]{Department of Mathematical Sciences, KAIST, South Korea}
\address[Jaehoon Kim]{Mathematics Institute, University of Warwick, United Kingdom}
\address[Hong Liu]{Mathematics Institute, University of Warwick, United Kingdom}
\email{dyk90@kaist.ac.kr}
\email{jaehoon.kim.1@warwick.ac.uk, mutualteon@gmail.com}
\email{h.liu.9@warwick.ac.uk}

\thanks{The first author was supported by the National Research Foundation of Korea (NRF) grant funded by the Korea government (MSIT) (No. NRF-2017R1A2B4005020) and also by TJ Park Science Fellowship of POSCO TJ Park Foundation. (D.~Kang). 
The second author  was supported  by the Leverhulme Trust Early Career Fellowship ECF-2018-538 (J.~Kim). 
The third author was supported  by the Leverhulme Trust Early Career Fellowship ECF-2016-52 (H.~Liu). }

\begin{abstract}
The extremal number $\ex(n,F)$ of a graph $F$ is the maximum number of edges in an $n$-vertex graph not containing $F$ as a subgraph. A real number $r \in [1,2]$ is realisable if there exists a graph $F$ with $\ex(n , F) = \Theta(n^r)$. Several decades ago, Erd\H{o}s and Simonovits conjectured that every rational number in $[1,2]$ is realisable. Despite decades of effort, the only known realisable numbers are $0,1, \frac{7}{5}, 2$, and the numbers of the form $1+\frac{1}{m}$, $2-\frac{1}{m}$, $2-\frac{2}{m}$ for integers $m \geq 1$. In particular, it is not even known whether the set of all realisable numbers contains a single limit point other than two numbers $1$ and $2$. 

In this paper, we make progress on the conjecture of Erd\H{o}s and Simonovits. First, we show that $2 - \frac{a}{b}$ is realisable for any integers $a,b \geq 1$ with $b>a$ and $b \equiv \pm 1 ~({\rm mod}\:a)$. This includes all previously known ones, and gives infinitely many limit points $2-\frac{1}{m}$ in the set of all realisable numbers as a consequence.

Secondly, we propose a conjecture on subdivisions of bipartite graphs. Apart from being interesting on its own, we show that, somewhat surprisingly, this subdivision conjecture in fact implies that every rational number between 1 and 2 is realisable.
\end{abstract}

\date{\today}
\maketitle

\section{Introduction}

\subsection{History and previous results}
 For a family of graphs $\mathcal{F}$, \emph{the extremal number} $\ex(n,\mathcal{F})$ is the maximum number of edges in an $n$-vertex graph which does not contain any subgraph isomorphic to a graph in $\mathcal{F}$. If $\mathcal{F}=\{F\}$, then we write $\ex(n,F)$ instead of $\ex(n,\mathcal{F})$. 

Since Mantel~\cite{mantel1907} determined the extremal number of a triangle in 1907, the study on the extremal number has been always at the core of extremal graph theory.
The classical Erd\H{o}s-Stone-Simonovits theorem~\cite{erdos-simonovits1966, erdos1946} showed that any $k$-chromatic graph $F$ satisfies
$$\ex(n,F) = \left (1-\frac{1}{k-1} + o(1) \right ) \binom{n}{2}.$$

While this provides good estimates for the extremal numbers of non-bipartite graphs, it only shows $\ex(n , F) = o(n^2)$ for any bipartite graph $F$.
Although there have been numerous attempts on finding better bounds of $\ex(n , F)$ for various bipartite graphs $F$, we know very little on the topic.  One of the fundamental conjectures on the subject is the following conjecture proposed by Erd\H{o}s and Simonovits.

\begin{conjecture}[Erd\H{o}s and Simonovits~\cite{erdos1981}]\label{conj:es'}
For every rational number $r \in [1,2]$, there exists a finite family $\mathcal{F}$ of graphs with $\ex(n , \mathcal{F}) = (c_{\mathcal{F}} + o(1)) n^r$ for some real number $c_{\mathcal{F}} > 0$.
\end{conjecture}

Many authors (see~\cite[Conjecture 5.1]{frankl86} and~\cite[Conjecture 2.37]{furedi2013}) stated a weaker version of Conjecture~\ref{conj:es'} that for every rational number $r \in [1,2]$, there exists a finite family $\mathcal{F}$ of graphs with $\ex(n , \mathcal{F}) = \Theta(n^r)$. In a recent breakthrough, this has been verified by Bukh and Conlon~\cite{bukh2018} using the random algebraic construction introduced by Bukh~\cite{bukh2015}.

\begin{theorem}[Bukh and Conlon \cite{bukh2018}]\label{thm:BC}
For every rational number $r \in [1,2]$, there exists a finite collection $\mathcal{F}$ of graphs with $\ex(n,\mathcal{F}) = \Theta(n^{r})$.
\end{theorem}

Recently, Fitch~\cite{2016fitch} showed that for any integer $k \geq 2$ and rational number $1 \leq r \leq k$, there exists a finite family $\mathcal{F}_k$ of $k$-uniform hypergraphs with $\ex(n , \mathcal{F}_k) = \Theta(n^r)$, extending Theorem~\ref{thm:BC} to uniform hypergraphs.





As a strengthening of Conjecture~\ref{conj:es'}, Erd\H{o}s and Simonovits (see~\cite[Section 8]{erdos1988}) also conjectured that for every rational number $r \in [1,2]$, there exists a graph $F$ with $\ex(n , F) = (c_F + o(1)) n^r$ for some real number $c_F > 0$. We state a slightly weaker version of their conjecture, which is our main interest.

\begin{conjecture}[Rational exponents conjecture~\cite{erdos1988}]\label{conj:es}
For every rational number $r \in [1,2]$, there exists a graph $F$ with $\ex(n , F) = \Theta(n^r)$.
\end{conjecture}

Let $r \in \mathbb{R}$ be \emph{realisable} (by $F$) if there exists a graph $F$ with $\ex(n,F) =\Theta(n^{r})$.
In contrast to the satisfying answer provided by Theorem~\ref{thm:BC}, Conjecture~\ref{conj:es} remains elusive.  Until now, the only known realisable numbers are $0, 1, \frac{7}{5}, 2$, the numbers of the form $1 + \frac{1}{m}$, $2- \frac{1}{m}$, $2-\frac{2}{m}$ with $m\in\mathbb{N}$. Faudree and Simonovits~\cite{faudree1983}, and Conlon~\cite{conlon2014} proved that the numbers $1+ \frac{1}{m}$ are realisable by Theta graphs $\theta_{m,\ell}$ for large $\ell$, which are graphs consisting of $\ell$ internally vertex-disjoint paths of length $m$ between two vertices (see \cite{verstraete2018, 2018Bukhtait} for recent progress on the extremal number of Theta graphs). 
K\"ovari, S\'os and Tur\'an~\cite{kovari1954}, and Alon, R\'onyai and Szab\'o~\cite{alon1999} proved that the numbers $2-\frac{1}{m}$ are realisable by unbalanced complete bipartite graphs (see also \cite{kollar1996, bukh2015}). 
Recently, Jiang, Ma and Yepremyan~\cite{2018jiang} proved that $2-\frac{2}{2m+1}$ is realisable by generalised cubes and $\frac{7}{5}$ is realisable by a so-called $3$-comb-pasting graph. Note that it is not even known whether there is a single limit point on the set of realisable numbers in the interval $(1,2)$.

\subsection{Our results}
One main contribution of this paper is the following theorem that provides infinitely many more realisable numbers, including all previously known realisable numbers.
\begin{theorem}\label{thm:main}
For each $a,b\in \mathbb{N}$ with $a< b$ and $b\equiv \pm 1 ~({\rm mod}\:a)$, the number $2- \frac{a}{b}$ is realisable.
\end{theorem}
As a consequence, this theorem provides infinitely many limit points on the set of realisable numbers.

\begin{corollary}\label{cor:limit}
For each $m\in \mathbb{N}$, the number $2 - \frac{1}{m}$ is a limit point in the set of realisable real numbers.
\end{corollary}

Secondly, we propose an approach to tackle Conjecture~\ref{conj:es} via the following conjecture on subdivision of graphs. For a graph $F$, let ${\rm sub}(F)$ be the \emph{$1$-subdivision} of $F$, obtained from $F$ by replacing all edges of $F$ with pairwise internally disjoint paths of length two.

\begin{conjecture}[Subdivision conjecture]\label{conj:subdiv}
Let $F$ be a bipartite graph. If $\ex(n , F) = O(n^{1+\alpha})$ for some $\alpha > 0$, then 
$$\ex(n , {\rm sub}(F)) = O(n^{1 + \frac{\alpha}{2}}).$$
\end{conjecture}

Apart from being interesting on its own, somewhat surprisingly, we show that this seemingly unrelated conjecture implies Conjecture~\ref{conj:es}.

\begin{theorem}\label{thm:subd-rational}
	If Conjecture~\ref{conj:subdiv} holds, then for every rational number $r \in [1,2]$, there exists a graph $F$ with $\ex(n , F) = \Theta(n^r)$.
\end{theorem}


It is worth noticing that if one considers instead 1-subdivision of non-bipartite $F$ in the Subdivision conjecture, then a stronger conclusion holds, as shown very recently by Conlon and Lee~\cite{2018conlonlee}. They proved that $\ex (n,{\rm sub}(K_t))=O(n^{3/2-\delta})$ for some $\delta=\delta(t)>0$. Nonetheless, the only known case for Conjecture~\ref{conj:subdiv} is when $F$ is a Theta graph. Indeed, for any $m\ge 2$,
$$\ex (n,\theta_{m,\ell})=O(n^{1+\frac{1}{m}})\quad \mbox{ and } \quad \ex (n,{\rm sub}(\theta_{m,\ell}))=\ex (n,\theta_{2m,\ell})=O(n^{1+\frac{1}{2m}}).$$

We do not know whether Conjecture~\ref{conj:subdiv} is true for complete bipartite graphs. Conlon and Lee~\cite{2018conlonlee} proved that the extremal number of the $1$-subdivision of $K_{s,t}$ is $O(n^{\frac{3}{2} - \frac{1}{12t}})$ when $t\geq s$. If Conjecture~\ref{conj:subdiv} is true, then this ought to be $O(n^{\frac{3}{2}-\frac{1}{2s}})$ for large $t$, where the exponent $\frac{3}{2} - \frac{1}{2s}$ only depends on the smaller number $s$ rather than $t$. To suggest the conjecture is plausible, we provide a proof that $\ex (n,{\rm sub}(K_{s,t}))=O(n^{\frac{3}{2} - \frac{1}{4s-2}})$, see Theorem~\ref{prop: subd complete bipartite} in the concluding remark section. Independent of our work, Janzer~\cite{Janzer2018} also proved the same bound for the 1-subdivision of complete bipartite graphs, and improved the upper bound of Conlon and Lee~\cite{2018conlonlee} for 1-subdivision of complete graphs.

\subsection{Organisation of the paper}
The paper is organised as follows. In Section~\ref{sec:prelim}, we will define several graphs, discuss the concept of balanced rooted graphs and collect several lemmas. In Section~\ref{sec:Dts}, we will prove part of Theorem~\ref{thm:main} that $2-\frac{a}{b}$ is realisable by a certain graph 
when $b \equiv -1~ ({\rm mod}\:a)$. In Section~\ref{sec:reduc}, we will prove Theorem~\ref{thm:subd-rational}, and finish the proof of Theorem~\ref{thm:main}, i.e.~$2-\frac{a}{b}$ is realisable when $b \equiv 1~ ({\rm mod}\:a)$ by using a combination of the reduction theorem of Erd\H{o}s and Simonovits \cite{erdos1970} and the theorem of Bukh and Conlon \cite{bukh2018}. Some concluding remarks are given in Section~\ref{sec-remark}.


\section{Preliminaries}
\label{sec:prelim}
\subsection{Basic terminology and lemmas}

Let $\mathbb{N}$ be the set of natural numbers. For any $n \in \mathbb{N}$, denote $[n]:=\{1,\dots, n\}$. We only consider finite simple graphs in this paper.
 For a graph $G$ and vertices $u,v \in V(G)$, we write ${\rm dist}(u,v)$ for the \emph{distance} between $u$ and $v$ in $G$, i.e.~the minimum number of edges in a path between $u$ and $v$. For a set $A\subseteq V(G)$ and $i\in \mathbb{N}\cup\{0\}$, let
$${\rm dist}(u,A) := \min_{v \in A} \{ {\rm dist}(u,v)\}\quad  \enspace \text{and} \enspace  \quad \Gamma_G^i (A) := \left \{u \in V(G) \: : \: \text{${\rm dist}(u,A) = i$} \right \}.$$
We denote \emph{the external neighbourhood} of $A$ to be $\Gamma_G(A) := \Gamma_G^1 (A)$, \emph{the common neighbourhood} of $A$ to be $N_G(A) := \bigcap_{a \in A}\Gamma_G(a)$, and \emph{the common degree} of $A$ to be $d_G (A) := |N_G(A)|$. For vertex sets $A,B \subseteq V(G)$, we let $\Gamma_G(A,B) := \Gamma_G(A) \cap B$, $N_G(A,B) := N_G(A) \cap B$, $d_G(A,B) := |N_G(A,B)|$, $E(A,B) := \left \{ \left \{a,b \right \} \: : \: a \in A, b \in B, ab \in E(G) \right \}$ and $e(A,B) := |E(A,B)|$. We also denote $E(A) := E(A,A)$ and $e(A) := e(A,A)$. For a set $A\subseteq V(G)$ and $s\in\mathbb{N}$, denote by ${A\choose s}$ all $s$-sets in $A$. We will omit the subscript $G$ if it is clear from the context.

We claim a result holds for $x\gg y$ if there exists an increasing function $f : [1,\infty) \to [1,\infty)$ such that the claimed result holds for all $x,y \geq 1$ with $x \geq f(y)$. We will not explicitly compute this function.
For convenience, we often omit the ceilings and floors and treats large number as integers if this does not affect the argument. We denote a star with $k$ edges a \emph{$k$-star} and the vertex of degree $k$ its \emph{centre}. If $k=1$, then we choose any of two vertices to be the centre. The following lemma is an easy consequence of Hall's theorem.
\begin{lemma}\label{lem:star}
Let $k \geq 1$ be an integer and $G$ be a bipartite graph with a bipartition $(A,B)$. If $k|S| \leq |\Gamma(S)|$ for any $S \subseteq A$, then $G$ contains vertex-disjoint $k$-stars whose centres cover $A$.
\end{lemma}


\subsection{Rooted blow-up of balanced bipartite graphs}\label{subsec:balanced}
Bukh and Conlon~\cite{bukh2018} introduced the following concepts of rooted blow-ups and balanced rooted trees. Here, we slightly extend their definitions. Consider a tuple $(F,R)$ of a graph $F$ and a proper subset $R\subsetneq V(F)$ of vertices. We say that the tuple $(F,R)$ is \emph{rooted} on $R$ and call $R$ the set of \emph{roots}. We simply write $F$ instead of $(F,R)$ if the roots are clear. For each non-empty set $S\subseteq V(F)$, let $\rho_F (S) := \frac{e_S}{|S|}$, where $e_S$ is the number of edges in $F$ incident with a vertex in $S$. Let $\rho(F):= \rho_F(V(F)\setminus R)$. Again, we omit the subscripts if it is clear. Note that $\rho(F)$ is well-defined as $R$ is a proper subset of $V(F)$.

We say that $(F,R)$ (or $F$ if $R$ is clear) is \emph{balanced} if $\rho_F (S) \geq \rho(F)$ holds for any non-empty subset $S\subseteq V(F)\setminus R$. For $\ell\in \mathbb{N}$ and a bipartite graph $F$ rooted on $R$, we let $F^{\ell}_R$ be the graph we obtain by taking disjoint union of copies of $F$ and identifying the vertices corresponding to a vertex $v$ into one vertex for each $v\in R$, see Figure~\ref{fig-blowup}. We omit the subscript $R$ if it is clear from the context.
If a graph $F$ is rooted on some set $R$, we will treat its $1$-subdivision ${\rm sub}(F)$ also as a rooted graph with the same set of roots $R$, see Figure~\ref{fig-blowup}.


\begin{figure}[h]
\centering
\begin{tikzpicture}[scale=0.7]

\draw
(1,2)--(-2,0)
(1,2)--(0,0)
(1,2)--(2,0)
(1,2)--(4,0)

(9+6,2)--(8+6,0)
(9+6,2)--(10+6,0)
(9+6,2)--(12+6,0)
(9+6,2)--(14+6,0)
(11+6,2)--(8+6,0)
(11+6,2)--(10+6,0)
(11+6,2)--(12+6,0)
(11+6,2)--(14+6,0)
(13+6,2)--(8+6,0)
(13+6,2)--(10+6,0)
(13+6,2)--(12+6,0)
(13+6,2)--(14+6,0)
;

\filldraw[fill=white] (1, 2) circle (2pt);
\filldraw[fill=black] (-2, 0) circle (2pt);
\filldraw[fill=black] (0, 0) circle (2pt);
\filldraw[fill=black] (2, 0) circle (2pt);
\filldraw[fill=black] (4, 0) circle (2pt);

\filldraw[fill=white] (9+6, 2) circle (2pt);
\filldraw[fill=white] (11+6, 2) circle (2pt);
\filldraw[fill=white] (13+6, 2) circle (2pt);
\filldraw[fill=black] (8+6, 0) circle (2pt);
\filldraw[fill=black] (10+6, 0) circle (2pt);
\filldraw[fill=black] (12+6, 0) circle (2pt);
\filldraw[fill=black] (14+6, 0) circle (2pt);

\node at (-2,-0.5) {$u_1$};
\node at (0,-0.5) {$u_2$};
\node at (2,-0.5) {$u_3$};
\node at (4,-0.5) {$u_4$};

\node at (8+6,-0.5) {$u_1$};
\node at (10+6,-0.5) {$u_2$};
\node at (12+6,-0.5) {$u_3$};
\node at (14+6,-0.5) {$u_4$};

\draw
(11-2,2)--(9.5-2,1)
(11-2,2)--(10.5-2,1)
(11-2,2)--(11.5-2,1)
(11-2,2)--(12.5-2,1)
(9.5-2,1)--(8-2,0)
(10.5-2,1)--(10-2,0)
(11.5-2,1)--(12-2,0)
(12.5-2,1)--(14-2,0)
;

\filldraw[fill=white] (11-2, 2) circle (2pt);
\filldraw[fill=white] (9.5-2, 1) circle (2pt);
\filldraw[fill=white] (10.5-2, 1) circle (2pt);
\filldraw[fill=white] (11.5-2, 1) circle (2pt);
\filldraw[fill=white] (12.5-2, 1) circle (2pt);
\filldraw[fill=black] (8-2, 0) circle (2pt);
\filldraw[fill=black] (10-2, 0) circle (2pt);
\filldraw[fill=black] (12-2, 0) circle (2pt);
\filldraw[fill=black] (14-2, 0) circle (2pt);

\node at (8-2,-0.5) {$u_1$};
\node at (10-2,-0.5) {$u_2$};
\node at (12-2,-0.5) {$u_3$};
\node at (14-2,-0.5) {$u_4$};

\node [rectangle, fill=none, draw=none, text width = 1cm] at (1.5,-1.5) {$S_4$};

\node [rectangle, fill=none, draw=none, text width = 1cm] at (9,-1.5) {${\rm sub}(S_4)$};

\node [rectangle, fill=none, draw=none, text width = 1cm] at (17.5,-1.5) {$S_4^3$};

\end{tikzpicture}
\caption{A 4-star $S_4$ rooted on its leaves, its $1$-subdivision and its blow-up.} \label{fig-blowup}
\end{figure}
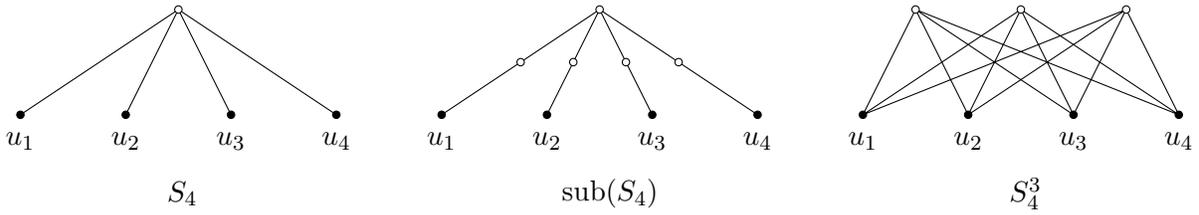

The following is a simple observation regarding balanced graphs. We omit its proof.
\begin{observation}\label{obs:easy}
Let $(F,R)$ be a balanced graph $F$ rooted on a non-empty 
set $R$. Then $(F^{\ell}_R,R)$ is balanced for all $\ell\in \mathbb{N}$. Moreover, if $F - E(R)$ is connected, then for any non-empty set $S\subseteq V(F)\setminus R$, we have $\rho_F(S)\geq 1$.
\end{observation}
%


For $a, b \in \mathbb{N}$ with $a-1\le b\le  2a-2$, consider an $a$-vertex path with non-root vertices labelled $1,\dots, a$ in order. Add $b-a+1$ root leaves, each adjacent to the following vertices on the path, respectively:
$$1, \left \lfloor 1+\frac{a}{b-a} \right \rfloor , \dots , \left \lfloor 1 + (b-a-1)\frac{a}{b-a} \right \rfloor, a.$$
Denote the resulting rooted tree by $T_{a,b}$ and define recursively $T_{a,b}$ for $b\ge 2a-1$ by adding one root leaf to each of the non-root vertices of $T_{a,b-a}$. It is proved in \cite{bukh2018} that $T_{a,b}$ is a balanced tree with $a$ non-root vertices and $b$ edges.




Bukh and Conlon~\cite{bukh2018} proved the following result that provides the lower bound of the extremal number of balanced bipartite rooted graphs.
\begin{lemma}[Bukh and Conlon~\cite{bukh2018}]\label{thm:lowerbound}
For every balanced bipartite rooted graph $F$ with $\rho(F)>0$, there exists a positive integer $\ell_0 = \ell_0(F)$ such that 
for all $\ell> \ell_0$, we have $\ex(n , F^\ell) = \Omega(n^{2 - \frac{1}{\rho(F)}})$.
\end{lemma}

Indeed, they stated Lemma~\ref{thm:lowerbound} only for balanced rooted trees $F$, but they did not use any assumption that $F$ is a tree. They also use the assumption that the set of root vertices to be independent. Nevertheless, we can consider a subgraph $F'$ of $F$ by removing edges between root vertices in $F$, which results in balanced bipartite rooted graphs with $\rho(F) = \rho(F')$, and apply the lemma to $F'$ to obtain the lower bound on $\ex(n , F^\ell)$.

Bukh and Conlon~\cite{bukh2018} also conjectured that for any balanced rooted tree $T$, there exists $\ell_0 = \ell_0(T)$ such that $\ex(n , T^\ell) = \Theta(n^{2 - \frac{1}{\rho(T)}})$ for all $\ell \geq \ell_0$. As Lemma~\ref{thm:lowerbound} shows that $\ex(n, T_{a,b}^{\ell}) = \Omega(n^{2 - \frac{a}{b}})$ for large $\ell$, their conjecture gives an approach to prove Conjecture~\ref{conj:es}. We remark that their conjecture cannot be generalised to balanced rooted bipartite graphs. Indeed, consider $F$ obtained from $T_{3,5}$ by identifying two root vertices attached on a first non-root vertex and a third non-root vertex on the path. The resulting graph $F$ contains $C_4$ as a subgraph, but $\ex(n , C_4) = \Theta(n^{3/2})$ while $\rho(F) = 5/3$.

For $s,t \in \mathbb{N}$, consider a $t$-star and attach $s$ leaves to each one of $t+1$ vertices of the $t$-stars. Let $D_{t,s}$ be the resulting tree rooted on its leaves. 
Note that $D_{1,s}$ and $D_{2,s}$ are isomorphic to
$T_{2,2s+1}$ and $T_{3,3s+2}$, respectively.  We call the centre of the original $t$-star  the \emph{centre} of $D_{t,s}$.
The graph $D_{t,s}$ is a tree with $(t+1)$ non-root vertices and $(s+1)(t+1)-1$ edges, and $\rho(D_{t,s}) = \frac{(t+1)(s+1)-1}{t+1}$.  Moreover, it is a balanced tree. Consider a blow-up $D_{t,s}^{\ell}$ of $D_{t,s}$. We call a vertex in $D_{t,s}^{\ell}$ a \emph{centre vertex} if it is a centre of a copy of $D_{t,s}$ in $D_{t,s}^{\ell}$ and we call a vertex in $D_{t,s}^{\ell}$ a \emph{core vertex} if it is a root vertex adjacent to all centre vertices.
We call a set of $s$ root vertices \emph{a cluster} if they are all adjacent to a same vertex in $D_{t,s}^{\ell}$. Note that the root vertices of $D_{t,s}^{\ell}$ partition into $t+1$ clusters, one of which contains all core vertices.  
\begin{figure}[h]
\centering
\begin{tikzpicture}[scale=0.8]

\draw
(0.5,3)--(0.6,2)
(1,3)--(0.6,2)
(1.5,3)--(0.6,2)
(0.5,3)--(1.4,2)
(1,3)--(1.4,2)
(1.5,3)--(1.4,2)

(0.6,2)--(-2.4,1)
(0.6,2)--(-0.4,1)
(0.6,2)--(1.6,1)
(0.6,2)--(3.6,1)
(1.4,2)--(-1.6,1)
(1.4,2)--(0.4,1)
(1.4,2)--(2.4,1)
(1.4,2)--(4.4,1)

(-2.4,1)--(-2.5,0)
(-2.4,1)--(-2,0)
(-2.4,1)--(-1.5,0)
(-1.6,1)--(-2.5,0)
(-1.6,1)--(-2,0)
(-1.6,1)--(-1.5,0)

(-0.4,1)--(-0.5,0)
(-0.4,1)--(0,0)
(-0.4,1)--(0.5,0)
(0.4,1)--(-0.5,0)
(0.4,1)--(0,0)
(0.4,1)--(0.5,0)

(1.6,1)--(1.5,0)
(1.6,1)--(2,0)
(1.6,1)--(2.5,0)
(2.4,1)--(1.5,0)
(2.4,1)--(2,0)
(2.4,1)--(2.5,0)

(3.6,1)--(3.5,0)
(3.6,1)--(4,0)
(3.6,1)--(4.5,0)
(4.4,1)--(3.5,0)
(4.4,1)--(4,0)
(4.4,1)--(4.5,0)
;
\filldraw[fill=black] (0.5, 3) circle (2pt);
\filldraw[fill=black] (1, 3) circle (2pt);
\filldraw[fill=black] (1.5, 3) circle (2pt);

\filldraw[fill=white] (0.6, 2) circle (2pt);
\filldraw[fill=white] (1.4, 2) circle (2pt);

\filldraw[fill=white] (-2.4, 1) circle (2pt);
\filldraw[fill=white] (-1.6, 1) circle (2pt);
\filldraw[fill=white] (-0.4, 1) circle (2pt);
\filldraw[fill=white] (0.4, 1) circle (2pt);
\filldraw[fill=white] (1.6, 1) circle (2pt);
\filldraw[fill=white] (2.4, 1) circle (2pt);
\filldraw[fill=white] (3.6, 1) circle (2pt);
\filldraw[fill=white] (4.4, 1) circle (2pt);

\filldraw[fill=black] (-2.5, 0) circle (2pt);
\filldraw[fill=black] (-2, 0) circle (2pt);
\filldraw[fill=black] (-1.5, 0) circle (2pt);

\filldraw[fill=black] (-0.5, 0) circle (2pt);
\filldraw[fill=black] (0, 0) circle (2pt);
\filldraw[fill=black] (0.5, 0) circle (2pt);

\filldraw[fill=black] (1.5, 0) circle (2pt);
\filldraw[fill=black] (2, 0) circle (2pt);
\filldraw[fill=black] (2.5, 0) circle (2pt);

\filldraw[fill=black] (3.5, 0) circle (2pt);
\filldraw[fill=black] (4, 0) circle (2pt);
\filldraw[fill=black] (4.5, 0) circle (2pt);

\node at (1,3.5) {$S_0$};
\node at (-2,-0.5) {$S_1$};
\node at (0,-0.5) {$S_2$};
\node at (2,-0.5) {$S_3$};
\node at (4,-0.5) {$S_4$};
\end{tikzpicture}
\caption{$D_{4,3}^2$ with core vertices $S_0$ and five clusters $S_0 , \dots , S_4$.}
\end{figure}
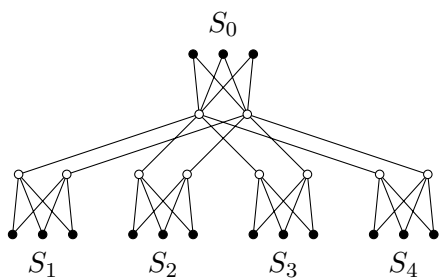

\begin{proposition}\label{prop:balanced}
For any $s,t\in \mathbb{N}$, the rooted tree $D_{t,s}$ is balanced.
\end{proposition}
\begin{proof}
Let $R$ be the set of leaves of $D_{t,s}$ which is precisely the root set of $D_{t,s}$.
Consider a non-empty set $S \subseteq V(D_{t,s}) \setminus R$. We have $1\leq |S|\leq t+1$.
If $S$ contains the centre of $D_{t,s}$, then 
$$\rho(S) = \frac{s|S|+t}{|S|} \geq \frac{s(t+1)+t}{t+1}=\rho(D_{t,s}).$$
If $S$ does not contain the centre, then 
$$\rho(S)  = \frac{(s+1)|S|}{|S|} = s + 1 \geq \rho(D_{t,s}).$$ Therefore, $D_{t,s}$ is balanced.
\end{proof}

\subsection{Dependent random choice and embedding  $D_{t,s}^{\ell}$}

The following variation of dependent random choice (Lemma~\ref{lem:drc}) together with the embedding lemma (Lemma~\ref{lem:embed}) will be useful for estimating $\ex(n,D_{t,s}^{\ell})$. For an excellent survey of dependent random choice, see~\cite{fox2011}.

\begin{lemma}\label{lem:drc}
Let $d,t\in \mathbb{N}$ and $G$ be a bipartite graph with
a vertex partition $(A,B)$.
If each vertex in $A$ has degree at least $d \geq \frac{2t |A|^{s-1}}{s!}$, then there exist a vertex $u\in B$ and a subset $A'\subseteq \Gamma_{G}(u,A)$ of size at least $\frac{d|A|}{2|B|}$ satisfying $d_{G}(S)\geq t$ for every $S\in \binom{A'}{s}$.
\end{lemma}
\begin{proof}
Choose a vertex $u \in B$ uniformly at random, and consider a set $X := \Gamma_G(u)\subseteq A$. For each $v \in A$, the probability that $v\in X$ is $\mathbb{P}(v \in X) = \frac{d_G(v)}{|B|} \geq \frac{d}{|B|}$. Hence we obtain
$\mathbb{E}[|X|] \geq \frac{d|A|}{|B|}.$

We say a set $S\in \binom{A}{s}$ of size $s$ is \emph{bad} if 
$d_{G}(S) < t$. Let $Y$ be the random variable indicating the number of bad sets in $\binom{X}{S}$. As $\mathbb{P}(S\subseteq X) = \frac{d_{G}(S) }{|B|}< \frac{t}{|B|}$, we have 
$$\mathbb{E}[Y] \leq \frac{t}{|B|}\cdot\left|\binom{X}{s} \right| \leq \frac{t |A|^{s}}{s!|B|}.$$ Let $X'$ be the set obtained from $X$ by deleting one element from each bad set $S\in \binom{X}{s}$, then $|X'| \geq |X|-Y$, and
$$\mathbb{E}[|X'|] \geq \mathbb{E}[|X|] - \mathbb{E}[Y] \geq \frac{d|A|}{|B|} - \frac{t |A|^{s}}{s! |B|} \geq \frac{d|A|}{2|B|}.$$
This implies that there exists a choice $A' \subseteq \Gamma_G(u)$ with $|A'| \geq \frac{d|A|}{2|B|}$ and $d_{G}(S)\geq t$ for each $S\in \binom{A'}{s}$, as desired.
\end{proof}

\begin{lemma}\label{lem:embed}
Let $s,t,\ell \in \mathbb{N}$ and $G$ be a graph. 
Let $W,A\subseteq V(G)$ be sets satisfying $|W|=s$ and $A=N_{G}(W)$. 
If $|A|\geq st+\ell$ and each $S\in \binom{A}{s+1}$ satisfies $|N_{G}(S)\setminus (A \cup W)| \geq \ell t$, then $G$ contains $D_{t,s}^{\ell}$ as a subgraph.
\end{lemma}
\begin{proof}
Recall that $D_{t,s}^{\ell}$ is obtained from the disjoint unions of $D_{t,s}$ by identifying corresponding leaves which are root vertices.
Map all $s$ core vertices into $W$. Further, we injectively map the remaining $st$ non-core root vertices and the $\ell$ centre vertices into $A$. This is possible as we have $|W|\geq s$ and $|A|\geq st+\ell$ with $W\cap A=\emptyset$.
Let $\psi$ be the injective function we have defined, which embeds all but $\ell t$ vertices of $D^{\ell}_{t,s}$ into $W\cup A$. 

Each vertex $v\in D^{\ell}_{t,s}$, with $\psi(v)$ not yet defined, is adjacent to $s$ root vertices and one centre vertex in $D^{\ell}_{t,s}$. As these $s+1$ neighbours of $v$ are injectively embedded in $A$, the set $S_v$ of their $\psi$-images is in $\binom{A}{s+1}$. Hence we have $|S'_v| \geq \ell t$ where $S'_v:= N_{G}(S_v)\setminus (A \cup W)$.
As there are $\ell t$ vertices $v$ for which $\psi(v)$ is not yet defined 
and $|S'_v|\geq \ell t$ holds for all such vertices $v$, we can choose $\psi(v)\in S'_v$ for all these vertices so that $\psi$ is still injective.
By the construction of $\psi$, it is easy to see that $\psi(D_{t,s}^{\ell})\subseteq G$. Hence $G$ contains $D_{t,s}^{\ell}$ as a subgraph.
\end{proof}

\section{The extremal number of $D_{t,s}^\ell$.}
\label{sec:Dts}
In this section, we prove the following theorem. Here, we write $D_{t-1,s-1}^{\ell}$ instead of $D_{t,s}^{\ell}$ only to make the formulas simpler.

\begin{theorem}\label{thm:mult_main}
Let $s,t\in \mathbb{N}\setminus \{1\}$. Then there exists $\ell_0$ such that for all $\ell \geq\ell_0$, we have $\ex(n, D_{t-1,s-1}^{\ell}) = \Theta(n^{2-\frac{t}{st-1}})$. 
\end{theorem}

As $D_{t-1,s-1}$ is balanced rooted graphs due to Observation~\ref{obs:easy} and Proposition~\ref{prop:balanced},  the following Lemma~\ref{lem:mult_main} together with Lemma~\ref{thm:lowerbound} implies Theorem~\ref{thm:mult_main}.

\begin{lemma}\label{lem:mult_main}
For all $\ell,s,t \in \mathbb{N}\setminus\{1\}$, we have $\ex(n , D_{t-1,s-1}^{\ell}) = O(n^{2 - \frac{t}{st-1}})$. 
\end{lemma}
\begin{proof}

As proved in  \cite{erdos1970}, any $m$-vertex graph with average degree $m^{\alpha}$ contains an $\widetilde{m}$-vertex graph with minimum degree at least $ \frac{1}{10\cdot 2^{\alpha^{-2}}} \widetilde{m}^{\alpha}$ with $\widetilde{m}\geq m^{\frac{\alpha(1-\alpha)}{1+\alpha}}$. Also any graph contains a spanning bipartite subgraph with the minimum degree at least the half of the original graph. Hence, it suffices to prove that for given $\ell,t$ and $s$ there exist $Q, n_0\in \mathbb{N}$ such that for any $n\geq n_0$, any $n$-vertex bipartite graph with minimum degree at least $Q n^{1- \frac{t}{st-1}}$ contains $D_{t-1,s-1}^{\ell}$ as a subgraph. 

Choose $n, q \in \mathbb{N}$ with $n\gg  q \gg \ell,s,t$ and let $d := n^{1- \frac{t}{st-1}}$. Then we have 
\begin{align}\label{eq:dn}
\frac{d^{s-1}}{n^{s-2}} = n^{\frac{t-1}{st-1}} \enspace \text{and} \enspace 
\frac{d^{s}}{n^{s-1}} = n^{-\frac{1}{st-1}}.
\end{align}
To derive a contradiction, assume that $G$ is an $n$-vertex bipartite graph satisfying $\delta(G) \geq 4q d$ that does not contain $D_{t-1,s-1}^{\ell}$ as a subgraph. Let $V:=V(G)$.

Recall that $D_{t-1,s-1}^{\ell}$ consists of $s-1$ core vertices, $\ell$ centre vertices, $(s-1)(t-1)$ non-core root vertices, and remaining $\ell(t-1)$ vertices that are neither roots nor centre. Also recall that the root vertices of $D_{t-1,s-1}^{\ell}$ partition into $t$ clusters each of which contains $s-1$ vertices.  \newline

\noindent{\bf Proof strategy.}
We first choose pairwise disjoint vertex sets $L_0, L_1, L_2$ and $L_3$ with $L_{i+1}\subseteq \Gamma_{G}(L_i)$ and $|L_{i+1}|$ is sufficiently larger than $|L_i|$.
We aim to embed the core vertices into $L_0$, the centre vertices into $L_1$, non-root neighbours of centre vertices into $L_2$ and the non-core root vertices to $L_3$. We let $S_1,\dots, S_{t-1}$ be the non-core clusters.  

\begin{figure}[h]
\centering
\begin{tikzpicture}[scale=0.8]
\draw
(0.5,3)--(0.6,2)
(1,3)--(0.6,2)
(1.5,3)--(0.6,2)
(0.5,3)--(1.4,2)
(1,3)--(1.4,2)
(1.5,3)--(1.4,2)

(0.6,2)--(-2.4,1)
(0.6,2)--(-0.4,1)
(0.6,2)--(1.6,1)
(0.6,2)--(3.6,1)
(1.4,2)--(-1.6,1)
(1.4,2)--(0.4,1)
(1.4,2)--(2.4,1)
(1.4,2)--(4.4,1)

(-2.4,1)--(-2.5,0)
(-2.4,1)--(-2,0)
(-2.4,1)--(-1.5,0)
(-1.6,1)--(-2.5,0)
(-1.6,1)--(-2,0)
(-1.6,1)--(-1.5,0)

(-0.4,1)--(-0.5,0)
(-0.4,1)--(0,0)
(-0.4,1)--(0.5,0)
(0.4,1)--(-0.5,0)
(0.4,1)--(0,0)
(0.4,1)--(0.5,0)

(1.6,1)--(1.5,0)
(1.6,1)--(2,0)
(1.6,1)--(2.5,0)
(2.4,1)--(1.5,0)
(2.4,1)--(2,0)
(2.4,1)--(2.5,0)

(3.6,1)--(3.5,0)
(3.6,1)--(4,0)
(3.6,1)--(4.5,0)
(4.4,1)--(3.5,0)
(4.4,1)--(4,0)
(4.4,1)--(4.5,0)
;
\filldraw[fill=black] (0.5, 3) circle (2pt);
\filldraw[fill=black] (1, 3) circle (2pt);
\filldraw[fill=black] (1.5, 3) circle (2pt);

\filldraw[fill=white] (0.6, 2) circle (2pt);
\filldraw[fill=white] (1.4, 2) circle (2pt);

\filldraw[fill=white] (-2.4, 1) circle (2pt);
\filldraw[fill=white] (-1.6, 1) circle (2pt);
\filldraw[fill=white] (-0.4, 1) circle (2pt);
\filldraw[fill=white] (0.4, 1) circle (2pt);
\filldraw[fill=white] (1.6, 1) circle (2pt);
\filldraw[fill=white] (2.4, 1) circle (2pt);
\filldraw[fill=white] (3.6, 1) circle (2pt);
\filldraw[fill=white] (4.4, 1) circle (2pt);

\filldraw[fill=black] (-2.5, 0) circle (2pt);
\filldraw[fill=black] (-2, 0) circle (2pt);
\filldraw[fill=black] (-1.5, 0) circle (2pt);

\filldraw[fill=black] (-0.5, 0) circle (2pt);
\filldraw[fill=black] (0, 0) circle (2pt);
\filldraw[fill=black] (0.5, 0) circle (2pt);

\filldraw[fill=black] (1.5, 0) circle (2pt);
\filldraw[fill=black] (2, 0) circle (2pt);
\filldraw[fill=black] (2.5, 0) circle (2pt);

\filldraw[fill=black] (3.5, 0) circle (2pt);
\filldraw[fill=black] (4, 0) circle (2pt);
\filldraw[fill=black] (4.5, 0) circle (2pt);

\node at (1,3.4) {\tiny Core vertices};
\node at (-4,3) {$L_0$};
\node at (-4,2) {$L_1$};
\node at (-4,1) {$L_2$};
\node at (-4,0) {$L_3$};
\node at (-2,-0.5) {$S_1$};
\node at (0,-0.5) {$S_2$};
\node at (2,-0.5) {$S_3$};
\node at (4,-0.5) {$S_4$};
\end{tikzpicture}
\caption{An embedding of $D_{4,3}^2$ with respect to the levels $L_0, L_1, L_2$ and $L_3$.}
\end{figure}
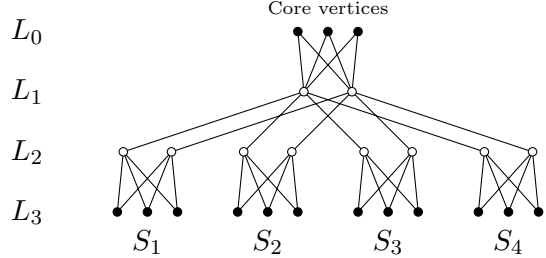

We will embed $S_1$ into $C_1\subseteq L_3$ in a nice manner that we can find $A_1\subseteq L_1$ and $B_1\subseteq L_2$ such that 
$A_1$ is a set of candidates for the images of centre vertices and $B_1$ is a set of candidates for the images of neighbours of $S_1$.
By repeatedly embedding $S_1,\dots, S_i$ into $C_1,\dots, C_i\subseteq L_3$ in an injective manner, we will find candidate sets $A_1\supseteq \dots \supseteq A_i$ for the images of the centre vertices, and the pairwise disjoint candidate sets $B_1, \dots, B_i \subseteq L_2$ for the neighbours of $S_1,\dots, S_i$. After embedding all $t-1$ clusters, if $|A_{t-1}|\geq \ell$ then this will give us a copy of $D_{t-1,s-1}^{\ell}$. \newline

\noindent {\bf Stage 1.}
We first choose a set $L_0$ of $s-1$ vertices which will be the images of the $s-1$ core vertices of $D_{t-1,s-1}^{\ell}$, and a set $L_1$ of vertices which are candidates for the images of the centre vertices of $D_{t-1,s-1}^{\ell}$. As $\delta(G)\geq 4q d$, we have
$$\sum_{L_0\in \binom{V}{s-1}} |N_{G}(L_0)| = \sum_{v\in V} \binom{d_{G}(v)}{s-1} \geq n \cdot \binom{4q d}{s-1}.$$
Hence, by averaging, there exists a vertex set $L_0 \in \binom{V}{s-1}$ with $d_{G}(L_0)\geq \binom{n}{s-1}^{-1} n\cdot \binom{4q d}{s-1} \geq \frac{q d^{s-1}}{n^{s-2}}\stackrel{\eqref{eq:dn}}{\geq} q  n^{\frac{t-1}{st-1}} $.
As $q\gg \ell$, we can choose a set $L_1 \subseteq N_{G}(L_0)$ with $|L_1| = \ell  n^{\frac{t-1}{st-1}}$.
\begin{claim}\label{cl: Gamma}
There exists a collection $\{\Gamma_u \subseteq N_{G}(u)\setminus (L_0\cup L_1) : u\in L_1\}$ of pairwise disjoint vertex sets with $|\Gamma_u| = 2d$.
\end{claim}
\begin{proof}
Note that $N_{G}(u)\cap L_1 =\emptyset$ for each $u\in L_1$  as $G$ is bipartite. We first show that $L_1$ expands: for each $A\subseteq L_1$, we have 
\begin{align}\label{eq:Hall cond}
|\Gamma_{G}(A)\setminus L_0| \geq 2d |A|.
\end{align} 
Suppose that $A\subseteq L_1$ satisfies $|B| < 2d|A|$ where
$B:=\Gamma_{G}(A)\setminus L_0$. 
Let $H$ be a bipartite graph with vertex partition $(A,B)$.
As $q\gg \ell,s,t\ge 2$, for any $v\in A$, we have
$$d_H(v,B)\ge \delta(G) -|L_0| \ge 4q n^{ \frac{(s-1)t-1}{st-1} } -s+1 \geq q n^{\frac{(t-1)(s-1)}{st-1} } \geq  \frac{2|L_1|^{s-1} \ell t}{s!} \geq \frac{2|A|^{s-1} \ell t}{s!}.$$
Hence, we can apply Lemma~\ref{lem:drc} to the bipartite graph $H$ with $4qd-s+1$ and $\ell t$ playing the roles of $d$ and $t$, respectively to obtain $A'\subseteq A$ with 
$$|A'| \geq \frac{|A|(4qd-s+1)}{2|B|} \geq \frac{|A|(4qd-s+1)}{ 2\cdot 2d|A|} \geq \ell + st,$$
such that any $S\in \binom{A'}{s}$ satisfies $d_{G}(S)\geq \ell t$.
We can then apply Lemma~\ref{lem:embed} to $G$ with $A', L_0$ and $s-1$ playing the role of $A, W$ and $s$, respectively to show that $G$ contains $D_{t-1,s-1}^{\ell}$ as a subgraph, a contradiction.
Hence, \eqref{eq:Hall cond} holds. Thus Lemma~\ref{lem:star} implies the existence of desired collection. This proves the claim.
\end{proof}

Let $L_2 := \bigcup_{u \in L_1}\Gamma_u$ and $L_3:= \Gamma_{G}(L_2)\setminus L_1$. As $G$ is bipartite, $L_0, L_1, L_2, L_3$ are pairwise disjoint vertex sets and the vertices in $L_2$ has no edges to $L_0$. Note that since $| L_1| =\ell n^{\frac{t-1}{st-1}}\leq qd$, each vertex $v\in L_2$ satisfies 
\begin{align}\label{eq: mindeg L3}
d_{G}(v, L_3) \geq 4qd- (qd+s-1) \geq 2qd.
\end{align}

\noindent {\bf Stage 2.}
Let $S_1,\dots, S_{t-1}$ be the sets of non-core clusters of $D_{t-1,s-1}^{\ell}$. We will embed these sets into sets $C_1,\dots, C_{t-1}$ in $L_3$. 
The following claim is useful for choosing the set $C_i$ so that we obtain candidate sets $A_i$ and $B_i$ of the correct sizes once we embedded $S_i$ into $C_i$.

\begin{claim}\label{claim:repetition}
Let $A^{\#}\subseteq L_1$, $B^*\subseteq L_2$ and $C^*\subseteq L_3$.
Suppose that $|C^*|\leq (s-1)(t-1)$ and for each $u\in A^{\#}$, we have $|\Gamma_u\cap B^*| \leq t-1$. Then there exist sets $A\subseteq A^{\#}$,
$B\subseteq L_2\setminus B^*$, $C \subseteq L_3\setminus C^*$ and a bijective function $f:A\rightarrow B$ satisfying the following.
\begin{itemize}
\item[$(\rm a)$] $|A|=|B| \geq n^{-\frac{1}{st-1}} |A^{\#}|$  and $|C|=s-1$.
\item[$(\rm b)$] $B\subseteq N_{G}(C)$.
\item[$(\rm c)$] $f(a) \in \Gamma_a$ for all $a\in A$.
\end{itemize}
\end{claim}
\begin{proof}
For each $u\in A^{\#}$, we consider the collection of $(s-1)$-tuples
$$\mathcal{C}_u:= \{ S\subseteq L_3\setminus C^* : |S|=s-1 \enspace \text{and} \enspace N_{G}(S)\cap (\Gamma_u\setminus B^*)\neq \emptyset \}.$$
We claim that for each $u\in A^{\#}$, we have 
\begin{align}\label{eq:mathCu}
|\mathcal{C}_u| \geq d^{s-1} |\Gamma_u\setminus B^*|.
\end{align}
Suppose $u\in A^{\#}$ and $|\mathcal{C}_u| <d^{s-1} |\Gamma_u\setminus B^*|$.
Let $X:= \Gamma_u \setminus B^*$.
Let $H_u$ be an auxiliary bipartite graph with a vertex partition $(X, \mathcal{C}_u)$ and 
$$E(H_u) =\{ wS \in X \times \mathcal{C}_u : w\in N_{G}(S)\}.$$

For each $w\in X$, by \eqref{eq: mindeg L3}, we have
$d_{G}(w, L_3\setminus C^*) = 2qd-|C^*|
\geq 2qd - st \geq qd.$
Since $|X|\le|\Gamma_u|\leq 2d$, we have
$$d_{H_u}(w) \geq \binom{qd}{s-1} \geq \frac{q d^{s-1}}{s^s}  
\geq \frac{2 |X|^{s-1} (\ell t)^{s}}{ s!}.$$ 
Hence, we can apply Lemma~\ref{lem:drc} to $H_u$ with $X$, $\mathcal{C}_u$ $\frac{ qd^{s-1}}{s^s}$ and $(\ell t)^s$ playing the roles of $A$, $B$, $d$ and $t$, respectively. Then we obtain $X' \subseteq \Gamma_{H_u}(S,X)\subseteq X$, where $S \in \mathcal{C}_u$ and 
$$|X'| \geq \frac{q d^{s-1}|X| }{ 2 s^s |\mathcal{C}_u|} \geq \frac{ q d^{s-1}|X|}{2 s^s d^{s-1} |\Gamma_u\setminus B^*|}
\geq (s-1)(t-1)+\ell$$ 
such that the following holds.
\begin{equation}
\begin{minipage}[c]{0.9\textwidth}\em
For any $U \in \binom{X'}{s}$, we have $d_{H_u}(U )\geq (\ell t)^s$. 
\end{minipage}
\end{equation}
Note that an $(s-1)$-set $S'\in N_{H_u}(U)$ if and only if all vertices $z\in S'$ are in $N_{G}(U)$.
Thus,
$$d_{H_u}(U) = \binom{|N_{G}(U, L_3\setminus C^* )|}{s-1}\geq (\ell t)^s,$$
implying that $d_{G}(U, L_3\setminus C^*)\geq \ell t$.
Hence, we can apply Lemma~\ref{lem:embed} to $G$ with $X', S, s-1$ and $t-1$ playing the roles of $A,W, s$ and $t$ to obtain a copy of $D_{t-1,s-1}^{\ell}$ in $G$,
a  contradiction. So \eqref{eq:mathCu} holds.

Now we aim to choose an appropriate $(s-1)$-set $C\subseteq L_3\setminus C^*$.
Let $H$ be an auxiliary bipartite graph with a vertex partition $(A^{\#} , \binom{L_3 \setminus C^*}{s-1})$ and 
$$ E(H):=\left\{ uS \in A^{\#}\times \binom{L_3 \setminus C^*}{s-1} : N_{G}(S)\cap (\Gamma_u\setminus B^*)\neq \emptyset \right\}.$$
In other words, $uS\in E(H)$ if $S\in \mathcal{C}_u$.
Claim~\ref{cl: Gamma} and \eqref{eq:mathCu} imply that
$$e(H) \geq \sum_{u\in A^{\#}} d^{s-1} |\Gamma_u\setminus B^*|
\geq  |A^{\#}|\cdot d^{s-1}\cdot (2d-t) \geq |A^{\#}|\cdot d^{s}.$$
Hence, by average, there exists a set $C\in \binom{|L_3 \setminus C^*|}{s-1}$ with 
$$d_H(C)\geq \binom{n}{s-1}^{-1}  e(H)\geq\binom{n}{s-1}^{-1}
|A^{\#}| d^{s} \geq \frac{d^{s}}{n^{s-1}} |A^{\#}|  \stackrel{\eqref{eq:dn}}{=} n^{-\frac{1}{st-1}} |A^{\#}|.$$

Let $A:= N_{H}(C)$. By the definition of $H$, 
  for each $a\in A$, there exists a vertex $f(a)\in \Gamma_u\setminus B^*$ such that $f(a) \in N_{G}(C)$. Let $B:=f(A)$. As $\Gamma_u\cap \Gamma_{u'}=\emptyset$ for distinct $u, u'\in A$, such a function $f$ is bijective between $A$ and $B$. It is easy to see that $A,B,C$ and $f$ satisfy properties (a)--(c).
\end{proof}

\begin{figure}[h]
\centering
\begin{tikzpicture}[scale=0.8]

\node (r1) at (0.5, 4) {};
\node (r2) at (1, 4) {};
\node (r3) at (1.5, 4) {};

\node (a11) at (-0.5, 2) {};
\node (a12) at (2.5, 2) {};

\node (a21) at (0.75 , 2) {};
\node (a22) at (2.25 , 2) {};

\node (b11) at (-2.5 , 0) {};
\node (b12) at (0.5 , 0) {};

\node (b21) at (2.25 , 0) {};
\node (b22) at (3.75 , 0) {};

\node (s11) at (-1.5, -2) {};
\node (s12) at (-1, -2) {};
\node (s13) at (-0.5, -2) {};

\node (s21) at (2.5, -2) {};
\node (s22) at (3, -2) {};
\node (s23) at (3.5, -2) {};

\fill[fill=gray!20] (r1.center)--(a11.center)--(a12.center);
\fill[fill=gray!30] (r2.center)--(a11.center)--(a12.center);
\fill[fill=gray!40] (r3.center)--(a11.center)--(a12.center);

\fill[fill=gray!20] (a11.center)--(a12.center)--(b12.center)--(b11.center);

\fill[fill=gray!20] (s11.center)--(b11.center)--(b12.center);
\fill[fill=gray!30] (s12.center)--(b11.center)--(b12.center);
\fill[fill=gray!40] (s13.center)--(b11.center)--(b12.center);

\begin{scope}
\clip
  (-1,1.5) rectangle (3,2.5);
  \draw[fill=gray!40] (1,2) ellipse [x radius=1.5,y radius=0.4];
\end{scope}

\begin{scope}
\clip
  (-2.5,-0.5) rectangle (0.5,0.5);
  \draw[fill=gray!40] (-1,0) ellipse [x radius=1.5,y radius=0.4];
\end{scope}

\fill[fill=gray!30] (a21.center)--(a22.center)--(b22.center)--(b21.center);

\fill[fill=gray!20] (s21.center)--(b21.center)--(b22.center);
\fill[fill=gray!30] (s22.center)--(b21.center)--(b22.center);
\fill[fill=gray!40] (s23.center)--(b21.center)--(b22.center);

\begin{scope}
\clip
  (0.75,1.5) rectangle (2.25,2.5);
  \draw[fill=gray!65] (1.5,2) ellipse [x radius=0.75,y radius=0.3];
\end{scope}

\begin{scope}
\clip
  (2.25,-0.5) rectangle (3.75,0.5);
  \draw[fill=gray!65] (3,0) ellipse [x radius=0.75,y radius=0.3];
\end{scope}

\filldraw[fill=black] (0.5, 4) circle (2pt);
\filldraw[fill=black] (1, 4) circle (2pt);
\filldraw[fill=black] (1.5, 4) circle (2pt);

\filldraw[fill=black] (-1.5, -2) circle (2pt);
\filldraw[fill=black] (-1, -2) circle (2pt);
\filldraw[fill=black] (-0.5, -2) circle (2pt);

\filldraw[fill=black] (2.5, -2) circle (2pt);
\filldraw[fill=black] (3, -2) circle (2pt);
\filldraw[fill=black] (3.5, -2) circle (2pt);

\node at (0.25,2) {\tiny $A_1$};
\node at (1.5,2) {\tiny $A_2$};
\node at (-1,0) {\tiny $B_1$};
\node at (3,0) {\tiny $B_2$};

\node at (1,4.4) {\tiny Core vertices};
\node at (-4,4) {$L_0$};
\node at (-4,2) {$L_1$};
\node at (-4,0) {$L_2$};
\node at (-4,-2) {$L_3$};
\node at (-1,-2.45) {\tiny $C_1$};
\node at (3,-2.45) {\tiny $C_2$};
\end{tikzpicture}
\caption{Embedding process of $D_{2,3}^\ell$ using Claim~\ref{claim:repetition}}
\end{figure}
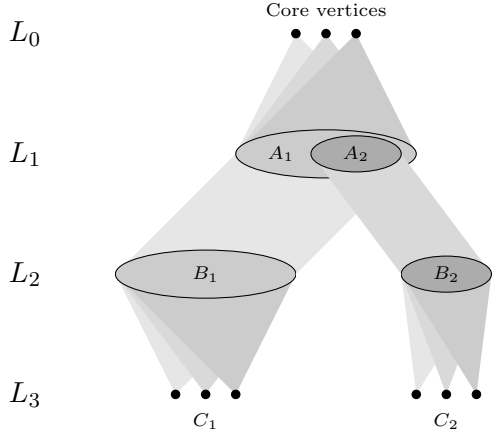

Let $A_0:= L_1$. 
For each $i= 1,2,\dots, t-1$ in order, we apply Claim~\ref{claim:repetition} with $A_{i-1}, \bigcup_{j=1}^{i-1}B_j$ and $\bigcup_{j=1}^{i-1}C_j$ playing the roles of $A^{\#}, B^*$ and $C^*$, respectively to obtain sets $A_i, B_i, C_i$ and $f_i$. This repetition is possible as 
the properties (a) and (c) ensures that $\bigcup_{j=1}^{i-1}B_j = \bigcup_{j=1}^{i-1} f_j(A_j)$ contains at most $i-1$ vertices in $\Gamma_u$ for each $u\in A_{i-1}$, as well as $|\bigcup_{j=1}^{i-1}C_j|\leq  (t-1)(s-1)$ holds.

Then we obtain sets $A_1\supseteq \dots \supseteq A_{t-1}$ and pairwise disjoint sets $B_1, \dots, B_{t-1}, C_{1},\dots, C_{t-1}$ and bijective functions $f_1,\dots, f_{t-1}$ with $f_i: A_i\rightarrow B_i$. Furthermore, for all $i\in [t-1]$ and $a\in A_i$, we have $|A_i| \geq n^{-\frac{i}{st-1}} |A_0| \geq \ell n^{\frac{t-1-i}{st-1}}$ and
$f_i(a)\in B_i\subseteq N_{G}(C_i)$. Moreover, as $A_1\supseteq \dots \supseteq A_{t-1}$, for each $i\in [t-1]$, the function $f_i$ is defined on each of the sets $A_{i+1},\dots, A_{t-1}$.

As $|A_{t-1}|\geq \ell$, we can choose a set $A$ of $\ell$ vertices in $A_{t-1}$.  Note that for each $i\in [t-1]$, the bipartite graph $G[A,f_i(A)]$ contains a perfect matching as we have $f_i(a) \in \Gamma_a$, and $G[C_i, f_i(A)]$ induces a complete bipartite graph $K_{s-1,\ell}$ as $f_i(A)\subseteq B_i \subseteq N_{G}(C_i)$. Since the sets $f_1(A)\subseteq B_{1},\dots, f_{t-1}(A)\subseteq  B_{t-1}$ are pairwise disjoint, the sets $L_0, A, f_1(A), \dots, f_{t-1}(A), C_1,\dots, C_{t-1}$ form a copy of $D_{t-1,s-1}^{\ell}$. More precisely, we can embed a copy of $D_{t-1,s-1}^{\ell}$ in such a way that the core vertices embed into $L_0$, centre vertices embed into $A$ and non-core root vertices embed into $C_1,\dots, C_{t-1}$. This proves the Lemma.
\end{proof}

\section{Reduction theorems}
\label{sec:reduc}
In this section, we will prove that in a certain class of bipartite graphs, the extremal number of a graph can be deduced from the extremal number of another simpler graph.

\subsection{Densification}
For $t\in \mathbb{N}$ and a connected bipartite graph $F$, let $(A,B)$ be its unique bipartition. We consider two disjoint set $R'_1$ and $R'_2$ of $t$ vertices disjoint from $V(F)$; and make the vertices of $R'_2$ adjacent to all vertices in $A$ and the vertices in $R'_1$ adjacent to all vertices in $B$; and add all possible edges between $R'_1$ and $R'_2$.
Let $F(t)$ denote the resulting graph.
If $F$ is a connected bipartite graph rooted on $R$, then we consider $F(t)$ as rooted on $R\cup R'_1\cup R'_2$ and let $F_*(t)$ denote the rooted graph we obtain from $F(t)$ by deleting all edges inside $R\cup R'_1\cup R'_2$, see Figure~\ref{fig-reduct}.

\begin{figure}[h]
\centering
\begin{tikzpicture}[scale=0.8]

\draw
(4,2)--(1,0)
(4,2)--(2,0)
(4,2)--(3,0)
(4,2)--(4,0)

(12,2)--(7,0)
(12,2)--(8,0)
(12,2)--(9,0)
(12,2)--(10,0)
(12,2)--(11,0)
(12,2)--(12,0)
(7,2)--(7,0)
(7,2)--(8,0)
(7,2)--(9,0)
(7,2)--(10,0)
(8,2)--(7,0)
(8,2)--(8,0)
(8,2)--(9,0)
(8,2)--(10,0)
(11,0)--(7,2)
(11,0)--(8,2)
(12,0)--(7,2)
(12,0)--(8,2)

(20,2)--(15,0)
(20,2)--(16,0)
(20,2)--(17,0)
(20,2)--(18,0)
(20,2)--(19,0)
(20,2)--(20,0)
(15,2)--(15,0)
(15,2)--(16,0)
(16,2)--(15,0)
(16,2)--(16,0)
;

\filldraw[fill=white] (4, 2) circle (2pt);
\filldraw[fill=white] (1, 0) circle (2pt);
\filldraw[fill=white] (2, 0) circle (2pt);
\filldraw[fill=black] (3, 0) circle (2pt);
\filldraw[fill=black] (4, 0) circle (2pt);

\filldraw[fill=white] (12, 2) circle (2pt);
\filldraw[fill=white] (7, 0) circle (2pt);
\filldraw[fill=white] (8, 0) circle (2pt);
\filldraw[fill=black] (9, 0) circle (2pt);
\filldraw[fill=black] (10, 0) circle (2pt);
\filldraw[fill=black] (11, 0) circle (2pt);
\filldraw[fill=black] (12, 0) circle (2pt);
\filldraw[fill=black] (7, 2) circle (2pt);
\filldraw[fill=black] (8, 2) circle (2pt);

\filldraw[fill=white] (20, 2) circle (2pt);
\filldraw[fill=white] (15, 0) circle (2pt);
\filldraw[fill=white] (16, 0) circle (2pt);
\filldraw[fill=black] (17, 0) circle (2pt);
\filldraw[fill=black] (18, 0) circle (2pt);
\filldraw[fill=black] (19, 0) circle (2pt);
\filldraw[fill=black] (20, 0) circle (2pt);
\filldraw[fill=black] (15, 2) circle (2pt);
\filldraw[fill=black] (16, 2) circle (2pt);

\node at (3,-0.5) {$w_1$};
\node at (4,-0.5) {$w_2$};
\node at (2.5,-1.2) {$F$};

\node at (9,-0.5) {$w_1$};
\node at (10,-0.5) {$w_2$};
\node at (11,-0.5) {$w_3$};
\node at (12,-0.5) {$w_4$};
\node at (7,2.5) {$w_5$};
\node at (8,2.5) {$w_6$};
\node at (9.5,-1.2) {$F(2)$};

\node at (17,-0.5) {$w_1$};
\node at (18,-0.5) {$w_2$};
\node at (19,-0.5) {$w_3$};
\node at (20,-0.5) {$w_4$};
\node at (15,2.5) {$w_5$};
\node at (16,2.5) {$w_6$};
\node at (17.5,-1.2) {$F_*(2)$};
\end{tikzpicture}
\caption{A connected bipartite graph $F$ with root set $R:=\{w_1,w_2\}$; and connected bipartite graphs $F(2), F_{*}(2)$ with root sets $R\cup R_1'\cup R_2'$, where $R_1':=\{w_3,w_4\}$ and $R_2':=\{w_5,w_6\}$.}\label{fig-reduct}
\end{figure}
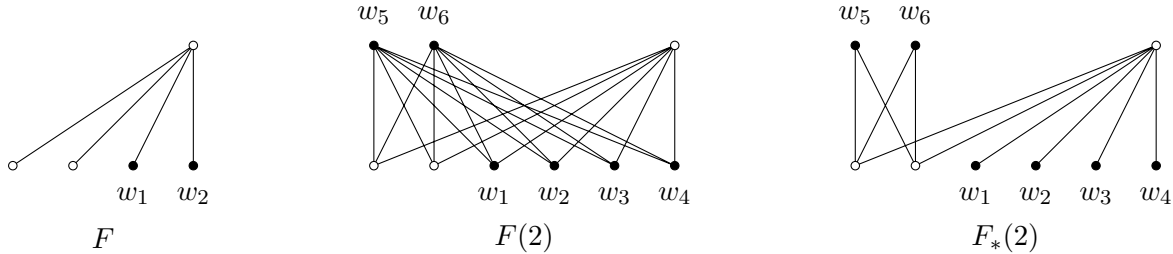

The following reduction theorem by Erd\H{o}s and Simonovits relates the extremal number of bipartite graphs $F$ and $F(t)$.

\begin{theorem}[Erd\H{o}s and Simonovits~\cite{erdos1970}]\label{thm:es_reduction}
Let $t\in \mathbb{N}$ and $F$ be a connected bipartite graph with $\ex(n , F) = O(n^{2-\alpha})$. Then $\ex(n , F(t)) = O(n^{2 - \beta})$ where $\beta^{-1} = \alpha^{-1} + t$.
\end{theorem}

Another important tools we use is Lemma~\ref{thm:lowerbound} by Bukh and Conlon. 
To be able to use Theorem~\ref{thm:es_reduction} and Lemma~\ref{thm:lowerbound} in the same framework, we need the following proposition.

\begin{proposition}\label{prop:dense_balanced}
Let $t\in \mathbb{N}$ and $F$ be a balanced rooted bipartite graph. Then both $F(t)$ and $F_*(t)$ are balanced rooted bipartite graph. 
\end{proposition}
\begin{proof}
As the edges between roots do not affect the definition of balancedness, it suffices to prove it for $F(t)$.
As every non-root vertices are adjacent to $t$ more vertices in $F(t)$ than $F$, it is easy to check that for any non-empty set $S\subseteq V(F)\setminus R$ we have $\rho_{F(t)}(S) = t+ \rho_F(S)$.
Hence, for every non-empty set $S$ of non-root vertices of $F(t)$, we have
$\rho(F(t)) = \rho(F)+t \leq \rho_F(S)+t = \rho_{F(t)}(S)$
Hence, $F(t)$ and $F_*(t)$ are balanced.
\end{proof}

We say that a number $r \in [1,2)$ is \emph{balancedly realisable by a graph $F$} if there exist a balanced connected rooted bipartite graph $F$ and a positive integer $\ell_0$ satisfying $\rho(F)=\frac{1}{2-r}$ and for every $\ell\geq \ell_0$, we have $\ex(n,F^{\ell})=\Theta(n^{r})$. By combining Theorem~\ref{thm:es_reduction} and Lemma~\ref{thm:lowerbound}, we can prove the following lemma.

\begin{lemma}\label{lem: densify}
For $a,b\in \mathbb{N}$ with $b>a$, if $2- \frac{a}{b}$ is balancedly realisable, then $2- \frac{a}{a+b}$ is also balancedly realisable.
\end{lemma}
\begin{proof}
By the assumption, there exist a balanced connected rooted bipartite graph $F$ and $\ell_0$ such that $\ex(n,F^{\ell})= \Theta( n^{2-\frac{a}{b}})$ for any $\ell\ge \ell_0$ and $\rho(F)=\frac{b}{a}$.

By Proposition~\ref{prop:dense_balanced}, $F_*(1)$ is also a balanced connected rooted bipartite graph with $\rho(F_*(1))= \rho(F)+1= \frac{a+b}{a}$. Hence, Lemma~\ref{thm:lowerbound} implies that there exists $\ell'_0\in \mathbb{N}$ such that for all $\ell \geq \ell'_0$, we have 
$$\ex(n, F_*(1)^{\ell}) = \Omega( n^{2- \frac{1}{\rho(F_*(1))} }) = \Omega( n^{2-\frac{a}{a+b}}).$$
On the other hand, as $F(1)^{\ell}=F^{\ell}(1)$ and $(\frac{a}{a+b})^{-1} = (\frac{a}{b})^{-1}+1$,
Theorem~\ref{thm:es_reduction} with the definition of $\ell_0$ implies that for all $\ell\geq \ell_0$,
$$\ex(n, F(1)^{\ell}) = \ex(n, F^{\ell}(1)) = O(n^{2-\frac{a}{a+b}}).$$

As $F_*(1)$ is a subgraph of $F(1)$, we have $\ex(n,F_*(1)^{\ell})\leq \ex(n,F(1)^{\ell})$. Thus, for any $\ell\geq \max\{\ell_0, \ell'_0\}$, we have
$$\ex(n, F_*(1)^{\ell}) \leq \ex(n,F(1)^{\ell}) = \Theta(n^{2-\frac{a}{a+b}}).$$
Therefore, $2-\frac{a}{a+b}$ is balancedly realisable.
\end{proof}

Now we are ready to prove Theorem~\ref{thm:main}.
\begin{proof}[Proof of Theorem~\ref{thm:main}]
Let $a\in \mathbb{N}$, it is known that the number $2- \frac{a}{a+1}$ is realisable by any large Theta graphs~\cite{faudree1983, conlon2014}, which is a blow-up of path rooted on the two end points. Hence $2- \frac{a}{a+1}$ is balancedly realisable. This with Lemma~\ref{lem: densify} implies that $2- \frac{a}{b}$ is realisable if $b>a$ and $b\equiv 1 ~({\rm mod}\:a)$. By Theorem 3.1, it also follows that $2 - \frac{a}{b}$ is realisable if $b>a$ and $b\equiv -1 ~({\rm mod}\:a)$, completing the proof.
\end{proof}

\subsection{Subdivision conjecture}

To see the motivation behind Conjecture~\ref{conj:subdiv}, suppose that $\ex(n,F) = O(n^{1+\alpha})$. Suppose that we have an $n$-vertex bipartite graph $G$ having no ${\rm sub}(F)$ as a subgraph with $e(G) = C n^{1+\alpha/2}$.
Consider an auxiliary graph $G^*$ with $V(G^*) = V(G)$ and $uv \in E(G^*)$ if and only if there exists a path of length two between $u$ and $v$.
By using a dependent random choice to the bipartite graph $G[\Gamma_{G}(v), \Gamma_{G}^2(v)]$ for each $v\in V(G)$, it is easy to see that $G^*$ contains at least $\Omega(n^{1+\alpha})$ edges, hence contains a copy of $F$. 
Note that this copy of $F$ will correspond to a (possibly degenerate) copy of ${\rm sub}(F)$ in $G$. Indeed, as $G^*$ contains many cliques of size $\Omega(n^{\alpha/2})$, namely $N_{G}(w)$ for each $w\in V(G)$, there is no guarantee that  the copies of $F$ is non-degenerate. However, it is plausible that a non-degenerate copy of ${\rm sub}(F)$ exists if $C$ is sufficiently large.


Conjecture~\ref{conj:es} and Conjecture~\ref{conj:subdiv} seem unrelated. However, much to our surprise, Conjecture~\ref{conj:subdiv} implies Conjecture~\ref{conj:es}. The rest of this section is devoted to show how two conjectures are connected.

\begin{proposition}\label{prop:sub balanced}
Given a balanced bipartite graph $F$ rooted on an independent set $R$ with
$\rho(F)\geq 1$, the $1$-subdivision ${\rm sub}(F)$ is also balanced rooted bipartite graph.
\end{proposition}
\begin{proof}
Let $R$ be the set of root vertices of $F$.
Let $b:= |E(F)|$ and $a:= |V(F)\setminus R|$. As $R$ is an independent set, we have $b\geq a$ as $\rho(F)= \frac{b}{a}\geq 1$.
For $S\subseteq V(F)$, let $e_S$ be the number of edges incident with a vertex in $S$ in the graph $F$ and let $e(S)$ be the number of edges whose both endpoints lie in $S$ in the graph $F$.

As $F$ is balanced, for each non-empty set $S\subseteq V(F)\setminus R$, we have 
\begin{align}\label{eq: rho F S}
\rho_F(S) = \frac{e_S}{|S|} \geq \frac{b}{a} = \rho(F)\ge 1.
\end{align}

Let $S\subseteq V({\rm sub}(F))\setminus R$. We aim to show $\rho_{{\rm sub}(F)}(S) \geq \frac{2b}{a+b} = \rho({\rm sub}(F))$. For each $i\in \{0,1,2\}$, let
$$S^*:= S\cap V(F) \enspace \text{and} \enspace S_i:= \{v \in S\setminus S^*: |N_{ {\rm sub}(F)}(v)\cap S^*| =i\}.$$
From these definitions, it is easy to see that the number of edges incident to $S$ in the graph ${\rm sub}(F)$ is
$e_{S^*} + e(S^*) + |S_1|+ 2|S_0|.$

If $S^*=\emptyset$, then $S$ is an independent set with each vertex having degree two, hence $\rho(S) = 2 \geq \frac{2b}{a+b} = \rho({\rm sub}(F)).$ Now we may assume $S^*\neq \emptyset$. Note that $S_1$ corresponds to a set of edges of $F$ incident with only one vertex of $S^*$, thus we have $e_{S^*} \geq e(S^*)  + |S_1|$. 
Also as $S_2$ corresponds to a set of edges whose both endpoints are in $S^*$, we have $|S_2|\leq e(S^*)$.
Thus, $|S_1|+|S_2| \leq e_{S^*}$. Together with $\frac{e_{S^*}}{|S^*|}\geq \rho(F)\ge 1$, we have
\begin{align}\label{eq: ratio}
\frac{ e_{S^*} + e(S^*) + |S_1|}{|S^*| + |S_1|+|S_2| } \geq 
\frac{ e_{S^*} + |S_2|+ |S_1|}{|S^*| + |S_1|+|S_2| } 
\geq
\frac{ e_{S^*}+e_{S^*} }{|S^*| +e_{S^*}} \stackrel{\eqref{eq: rho F S}}{\geq} \frac{2b}{a+b}.
\end{align}
Then, we have
\begin{eqnarray*}
\rho_{{\rm sub}(F)}(S) &=& 
\frac{e_{S^*} + e(S^*) + |S_1|+ 2|S_0|   }{|S^*|  + |S_1|+|S_2| + |S_0|} \geq \frac{e_{S^*} + e(S^*) + |S_1| + \frac{2b}{a+b}|S_0| }{ |S^*|  + |S_1| +|S_2| +|S_0| } \\
&\stackrel{\eqref{eq: ratio}}{\geq} & \frac{\frac{2b}{a+b}|S\setminus S_0| + \frac{2b}{a+b}|S_0| }{ |S\setminus S_0| + |S_0| } = \frac{2b}{a+b} = \rho({\rm sub}(F)).
\end{eqnarray*}
This proves the proposition.
\end{proof}

Let $\mathcal{F}_0$ be the minimal collection of balanced connected rooted bipartite graphs satisfying the following.
\begin{itemize}
\item $\mathcal{F}_0$ includes all stars rooted on the leaves;
\item $\mathcal{F}_0$ is closed under taking $1$-subdivision, i.e.~if $F \in \mathcal{F}_0$, then ${\rm sub}(F)\in \mathcal{F}_0$;
\item If $F \in \mathcal{F}_0$, then $F_{\ast}(1)\in \mathcal{F}_0$.
\end{itemize}

Note that Observation~\ref{obs:easy}, Propositions~\ref{prop:dense_balanced} and~\ref{prop:sub balanced} guarantee that every $F\in \mathcal{F}_0$ is bipartite, balanced and connected, and $\rho(F) \geq 1$. Moreover, for every rooted bipartite graph $(F,R)\in \mathcal{F}_0$, the root set $R$ is always an independent set.

\begin{lemma}\label{lem: sparsify}
Suppose for any $F \in \mathcal{F}_0$, there exists $\ell_0 = \ell_0(F)$ such that Conjecture~\ref{conj:subdiv} holds for $F^{\ell}$ for all $\ell\geq \ell_0$. If $a,b\in \mathbb{N}$, $b>a$, are such that $2 - \frac{a}{b}$ is balancedly realisable by a graph in $\mathcal{F}_0$, then $2- \frac{a+b}{2b}$ is also balancedly realisable by a graph in $\mathcal{F}_0$.
\end{lemma}
\begin{proof}
By the assumption, there exist a balanced connected rooted bipartite graph $F\in \mathcal{F}_0$ and $\ell_0$ such that $\rho(F)=\frac{b}{a}$ and $\ex(n,F^{\ell})= \Theta( n^{2-\frac{a}{b}})= \Theta( n^{1 + \frac{b-a}{b}})$ for all $\ell\ge \ell_0$.

By Proposition~\ref{prop:sub balanced}, ${\rm sub}(F)\in \mathcal{F}_0$ is also a balanced connected rooted bipartite graph with $\rho({\rm sub}(F)) = \frac{2b}{a+b}$. Hence, Lemma~\ref{thm:lowerbound} implies that there exists some $\ell_1\in \mathbb{N}$ such that, for all $\ell \geq \ell_1$ we have  
$$
\ex(n, {\rm sub}(F)^{\ell }) = \Omega( n^{2- \frac{1}{\rho({\rm sub}(F))} }) = \Omega( n^{2-\frac{a+b}{2b}}).
$$
On the other hand, by assumption, Conjecture~\ref{conj:subdiv} holds for $F^{\ell}$ for all $\ell\ge \ell_0$, i.e.
$$\ex(n, {\rm sub}(F^{\ell})) = O( n^{1+ \frac{b-a}{2b}}) = O(n^{2-\frac{a+b}{2b}}).$$
Note that the root set of $F\in \mathcal{F}_0$ is an independent set, so taking $1$-subdivision of a rooted blow-up of $F$ is the same as taking a rooted blow-up of the $1$-subdivision of $F$, that is, ${\rm sub}(F^{\ell}) = {\rm sub}(F)^{\ell}$. 
Thus, for any $\ell\ge \max\{\ell_0,\ell_1 \}$, 
$\ex(n, {\rm sub}(F)^{\ell })=\Theta( n^{2-\frac{a+b}{2b}})$. Consequently, $2- \frac{a+b}{2b}$ is balancedly realisable by the graph $ {\rm sub}(F)\in \mathcal{F}_0$.
\end{proof}

Now we are ready to prove Theorem~\ref{thm:subd-rational}. In fact, a weaker version of Conjecture~\ref{conj:subdiv} already implies Conjecture~\ref{conj:es} as follows.

\begin{theorem}\label{thm:e-s}
Suppose that for each $F\in \mathcal{F}_0$, there exists $\ell_0=\ell_0(F)$ such that Conjecture~\ref{conj:subdiv} holds for $F^{\ell}$ for all $\ell\geq \ell_0$. Then Conjecture~\ref{conj:es} holds.
\end{theorem}
\begin{proof}
We will show that for all $a,b\in \mathbb{N}$ with $a< b$, the number $2- \frac{a}{b}$ is balancedly realisable under the assumption of theorem. Note that unbalanced complete bipartite graph (which is a blow-up of a star rooted on its leaves) shows that the number $2-\frac{a}{b}$ is balancedly realisable by a graph in $\mathcal{F}_0$ for $a=1$. We use induction on $a+b$.
Assume that $(a,b)$ is a minimum counterexample. 

If $b> 2a$, then $b-a> a$ and $a+(b-a) < a+b$. Hence,
by the induction hypothesis, $2- \frac{a}{b-a}$ is balancedly realisable by a graph $F\in \mathcal{F}_0$ and Lemma~\ref{lem: densify} implies that $2-\frac{a}{b}$ is also balancedly realisable by $F_*(1)\in \mathcal{F}_0$.

If $a<b<2a$, then $(2a-b)+b  < a+b$ and $2a-b\geq 1$. Hence, by the induction hypothesis, $2 - \frac{2a-b}{b}$ is balancedly realisable by a graph in $\mathcal{F}_0$ and Lemma~\ref{lem: sparsify} implies that $2 -\frac{2a-b+b}{2b} = 2-\frac{a}{b}$ is balancedly realisable by a graph in $\mathcal{F}_0$. Hence, $2-\frac{a}{b}$ is balancedly realisable for all natural numbers $b>a$.
As $1$ and $2$ are trivially realisable, this shows that every rational number $r \in [1,2]$ is realisable if the assumption of Theorem~\ref{thm:e-s} is true.
\end{proof}

\section{Concluding Remarks}\label{sec-remark}

\subsection{Bipartite graphs with large radius}

Our results provide infintely many realisable numbers most of which are somewhat closer to 2 than 1. The reason for this is that the graph $D_{t,s}^{\ell}$ we considered has  radius two and gets denser as we increase the parameter $t$ and $s$, and Lemma~\ref{lem: densify} also produces new realisable number which is bigger than the original. Hence, to attack Conjecture~\ref{conj:es}, we need a method to deal with sparse graphs.

One obvious way to is to prove Conjecture~\ref{conj:subdiv} for blow-ups of balanced rooted bapartite graphs.
As Theorem~\ref{thm:subd-rational}  suggests, this implies Conjecture~\ref{conj:es}. Another natural way to pursue is to consider a balanced tree with large radius, and study its blow-up. Towards this direction, 
we are only able to extend our method slightly to obtain the following result, regarding a blow-up of a balanced tree with radius three. Note that $\frac{10}{7}$ does not provide new realisable sequence as Theorem~\ref{thm:mult_main} shows this is also realisable by $D_{3,1}^{\ell}$. We include its proof in the appendix.

\begin{theorem}\label{thm-T47}
	There exists $\ell_0\in \mathbb{N}$ such that for all $\ell>\ell_0$, we have $\ex(n , T_{4,7}^\ell) = \Theta(n^{10/7})$.
\end{theorem}


It would be interesting to generalise Theorem~\ref{thm:mult_main} as follows. For $s, t\in \mathbb{N}$, consider the following balanced tree with large radius. Let $H_{t,s}$ be the rooted tree obtained from a $t$-star by subdividing each edge $s$ times and by attaching a leaf to the centre of the $t$-star; the root set of $H_{t,s}$ is its leaf-set. 
It is easy to check that $H_{t,s}$ is a balanced tree with $\rho(H_{t,s})= \frac{1+ (s+1)t}{ 1+ st}$. The following  seems plausible.

\begin{figure}[h]
\centering
\begin{tikzpicture}[scale=0.8]

\draw
(0,3)--(1,3)
(1,3)--(2.5,2)
(1,3)--(2.5,2.5)
(1,3)--(2.5,3)
(1,3)--(2.5,3.5)
(1,3)--(2.5,4)

(2.5,2)--(3.5,2)
(2.5,2.5)--(3.5,2.5)
(2.5,3)--(3.5,3)
(2.5,3.5)--(3.5,3.5)
(2.5,4)--(3.5,4)

(4.5,2)--(5.5,2)
(4.5,2.5)--(5.5,2.5)
(4.5,3)--(5.5,3)
(4.5,3.5)--(5.5,3.5)
(4.5,4)--(5.5,4)

(5.5,2)--(6.5,2)
(5.5,2.5)--(6.5,2.5)
(5.5,3)--(6.5,3)
(5.5,3.5)--(6.5,3.5)
(5.5,4)--(6.5,4)
;

\filldraw[fill=black] (0, 3) circle (2pt);
\filldraw[fill=white] (1, 3) circle (2pt);

\filldraw[fill=white] (2.5, 2) circle (2pt);
\filldraw[fill=white] (2.5, 2.5) circle (2pt);
\filldraw[fill=white] (2.5, 3) circle (2pt);
\filldraw[fill=white] (2.5, 3.5) circle (2pt);
\filldraw[fill=white] (2.5, 4) circle (2pt);

\filldraw[fill=white] (3.5, 2) circle (2pt);
\filldraw[fill=white] (3.5, 2.5) circle (2pt);
\filldraw[fill=white] (3.5, 3) circle (2pt);
\filldraw[fill=white] (3.5, 3.5) circle (2pt);
\filldraw[fill=white] (3.5, 4) circle (2pt);

\filldraw[fill=white] (4.5, 2) circle (2pt);
\filldraw[fill=white] (4.5, 2.5) circle (2pt);
\filldraw[fill=white] (4.5, 3) circle (2pt);
\filldraw[fill=white] (4.5, 3.5) circle (2pt);
\filldraw[fill=white] (4.5, 4) circle (2pt);

\filldraw[fill=white] (5.5, 2) circle (2pt);
\filldraw[fill=white] (5.5, 2.5) circle (2pt);
\filldraw[fill=white] (5.5, 3) circle (2pt);
\filldraw[fill=white] (5.5, 3.5) circle (2pt);
\filldraw[fill=white] (5.5, 4) circle (2pt);

\filldraw[fill=black] (6.5, 2) circle (2pt);
\filldraw[fill=black] (6.5, 2.5) circle (2pt);
\filldraw[fill=black] (6.5, 3) circle (2pt);
\filldraw[fill=black] (6.5, 3.5) circle (2pt);
\filldraw[fill=black] (6.5, 4) circle (2pt);

\node at (4 , 2) {$\dots$};
\node at (4 , 2.5) {$\dots$};
\node at (4 , 3) {$\dots$};
\node at (4 , 3.5) {$\dots$};
\node at (4 , 4) {$\dots$};

\tikzset{
    position label/.style={
       below = 3pt,
       text height = 1.5ex,
       text depth = 1ex
    },
   brace/.style={
     decoration={brace, mirror},
     decorate
   }
}

\node (r1) at (6.8,2) {};
\node (r2) at (6.8,4) {};
\node (s1) at (2.5 , 1.7) {};
\node (s2) at (5.5 , 1.7) {};
\node at (7.5,3) {$t$};
\node at (4 , 1) {$s$};

\draw [brace] (r1.east) -- (r2.east);
\draw [brace] (s1.south) -- (s2.south);

\end{tikzpicture}
\caption{$H_{t,s}$}
\end{figure}

\begin{problem}\label{conj:longtree}
For any positive integers $s$ and $t$, there exists $\ell_0 = \ell_0(s,t)$ such that
$$\ex(n , H_{t,s}^{\ell}) = \Theta(n^{1 + \frac{t}{1 + (s + 1)t} })$$
for any integer $\ell \geq \ell_0$.
\end{problem}

If this is true, then it would provide infinitely many new limit points $1+ \frac{1}{m}$ in the set of realisable number and this together with Lemma~\ref{lem: densify} would provide more realisable numbers.
The method we used in Lemma~\ref{lem:mult_main} cannot be directly  generalised to this problem. In particular, we need to prove that $i$-th neighbourhood of a vertex has size proportional to the $i$-th power of the average degree of $G$. This seems difficult to prove without a major improvement of the method.

\subsection{The $1$-subdivision of complete bipartite graphs.}
We may consider the 1-subdivision of complete bipartite graphs as an example of Conjecture~\ref{conj:subdiv}. If Conjecture~\ref{conj:subdiv} is true, then 
$$\ex(n,{\rm sub}(K_{s,t})) = \Theta(n^{\frac{3}{2} - \frac{1}{2s} })$$ 
must hold for large $t$, where the lower bound is obtained from Lemma~\ref{thm:lowerbound}. The best known upper bound is  
by Conlon and Lee~\cite{2018conlonlee} $\ex(n,{\rm sub}(K_{s,t})) \leq O(n^{\frac{3}{2} - \frac{1}{12t}})$.
Improving their result, we are able to prove the following proposition with an exponent depending only on $s$. We remark that Janzer~\cite{Janzer2018} independently obtained the same result.

\begin{proposition}\label{prop: subd complete bipartite}
	For $t,s\in \mathbb{N}$ with $t \geq s$, we have
	$\ex(n,{\rm sub}(K_{s,t})) \leq O(n^{\frac{3}{2} - \frac{1}{4s-2}})$.
\end{proposition}
\begin{proof}[Proof sketch]
	Note that ${\rm sub}(K_{s,t})$ is a subgraph of $D_{s,1}^{t}$. Hence a direct application of Theorem~\ref{thm:mult_main} implies $\ex(n,{\rm sub}(K_{s,t})) \leq O(n^{\frac{3}{2} - \frac{1}{4s+2}})$. 
	To improve the number $\frac{1}{4s+2}$ to $\frac{1}{4s-2}$, we can follow the proof of Theorem~\ref{thm:mult_main}  with the following minor modifications.
	
	Let us choose the numbers $n\gg q \gg s,t$.
	We may assume that $G$ is a bipartite graph with $n$ vertices and the minimum degree $4qd$ with $ d := n^{1 - \frac{1}{4s-2}}$. We choose an arbitrary vertex $u_0 \in V(G)$ and $L_{-1} := \{ u_0 \}$. Let $L_0 \subseteq \Gamma_G (u_0)$ with $|L_0| = qd$. By using Hall's theorem, we can find a set $L_1 \subseteq \Gamma_G (L_0) \setminus \{ r_0 \}$ and a perfect matching on $G[L_0,L_1]$. For each $v \in L_0$, let $g(v) \in L_1$ be the vertex adjacent to $v$ in the matching.
	
	For each $w \in L_1$, we can find a collection of pairwise disjoint sets $\{ \Gamma_{w} \subseteq N_G(w)\setminus(L_0\cup L_1) : w\in L_1 \}$ with $|\Gamma_{w}| \geq 2d$. Let $L_2 := \bigcup_{w \in L_1} \Gamma_w$ and $L_3:= \Gamma_G(L_2)\setminus (L_0\cup L_1)$. Now the rest of the proof follows Stage 2 in the proof of Theorem~\ref{thm:mult_main} with $s=2$ and $\ell=t$, except that we have to use vertices of $g^{-1}(A)$ and $u_0$ at the end to obtain a copy of ${\rm sub}(K_{s,t})$.
\end{proof}

\subsection{Phase transition with respect to the number of blown-up copies}
It would be interesting to determine for what $\ell$ the order of magnitude changes in Theorem~\ref{thm:mult_main}. Note that the upper bound 
$$\ex(n , D_{t-1,s-1}^{\ell}) = O(n^{2 - \frac{t}{st-1}})$$ is not tight when $\ell$ is small. Indeed, it is known that 
$$\ex(n , D_{2,1}^2) = \Theta(n^{4/3}) \quad \text{and} \quad \ex(n , D_{2,1}^{\ell}) = \Theta(n^{7/5})$$
for $\ell$ sufficiently large. Indeed, for $D_{2,1}^2$, the lower bound follows from the fact that $D_{2,1}^2$ contains $C_6$ as a subgraph, and the upper bound follows from a reduction theorem of Faudree and Simonovits~\cite{faudree1983}.

Thus, for the blow-up of the graph $D_{2,1}$, the transition happens when the number of copies $\ell$ is larger than 2. This is in contrast to the well-known conjecture for even-cycles, stating that $\ex (n, C_{2k})=\Theta(n^{1+1/k})$. Indeed, even-cycles are theta graphs with two disjoint paths, and $\ex (n,\theta_{k,\ell})=\Theta(n^{1+1/k})$ for large $\ell$. So the even-cycles conjecture suggests that for paths rooted at leaves, the transition happens already at $\ell=2$. Recently, Verstra\"ete and Williford~\cite{verstraete2018} showed that $\ex(n , \theta_{4,3}) = \Theta(n^{5/4})$, giving an evidence to the even-cycle conjecture for $C_8$.

\bibliographystyle{abbrv}
\bibliography{references-ecd}

\begin{thebibliography}{10}

\bibitem{alon1999}
N.~Alon, L.~R\'onyai, and T.~Szab\'o.
\newblock Norm-graphs: variations and applications.
\newblock {\em J. Combin. Theory Ser. B}, 76(2):280--290, 1999.

\bibitem{bukh2015}
B.~Bukh.
\newblock Random algebraic construction of extremal graphs.
\newblock {\em Bull. Lond. Math. Soc.}, 47(6):939--945, 2015.

\bibitem{bukh2018}
B.~Bukh and D.~Conlon.
\newblock Rational exponents in extremal graph theory.
\newblock {\em J. Eur. Math. Soc.}, 20:1747--1757, 2018.

\bibitem{2018Bukhtait}
B.~{Bukh} and M.~{Tait}.
\newblock {Tur\'{a}n number of theta graphs}.
\newblock {\em arXiv preprint arXiv:1804.10014}, Apr. 2018.

\bibitem{conlon2014}
D.~Conlon.
\newblock Graphs with few paths of prescribed length between any two vertices.
\newblock {\em Bull. Lond. Math. Soc.}, 2014.
\newblock to appear.

\bibitem{2018conlonlee}
D.~{Conlon} and J.~{Lee}.
\newblock {On the extremal number of subdivisions}.
\newblock {\em arXiv preprint arXiv:1807.05008}, July 2018.

\bibitem{erdos1981}
P.~Erd\H{o}s.
\newblock On the combinatorial problems which {I} would most like to see
  solved.
\newblock {\em Combinatorica}, 1(1):25--42, 1981.

\bibitem{erdos1988}
P.~Erd\H{o}s.
\newblock Problems and results in combinatorial analysis and graph theory.
\newblock In {\em Proceedings of the {F}irst {J}apan {C}onference on {G}raph
  {T}heory and {A}pplications ({H}akone, 1986)}, volume~72, pages 81--92, 1988.

\bibitem{erdos-simonovits1966}
P.~Erd\H{o}s and M.~Simonovits.
\newblock A limit theorem in graph theory.
\newblock {\em Studia Sci. Math. Hungar.}, 1:51--57, 1966.

\bibitem{erdos1970}
P.~Erd\H{o}s and M.~Simonovits.
\newblock Some extremal problems in graph theory.
\newblock pages 377--390, 1970.

\bibitem{erdos1946}
P.~Erd\H{o}s and A.~H. Stone.
\newblock On the structure of linear graphs.
\newblock {\em Bull. Amer. Math. Soc.}, 52:1087--1091, 1946.

\bibitem{faudree1983}
R.~J. Faudree and M.~Simonovits.
\newblock On a class of degenerate extremal graph problems.
\newblock {\em Combinatorica}, 3(1):83--93, 1983.

\bibitem{2016fitch}
M.~{Fitch}.
\newblock {Rational exponents for hypergraph Turan problems}.
\newblock {\em arXiv preprint arXiv:1607.05788}, July 2016.

\bibitem{fox2011}
J.~Fox and B.~Sudakov.
\newblock Dependent random choice.
\newblock {\em Random Structures Algorithms}, 38(1-2):68--99, 2011.

\bibitem{frankl86}
P.~Frankl.
\newblock All rationals occur as exponents.
\newblock {\em J. Combin. Theory Ser. A}, 42(2):200--206, 1986.

\bibitem{furedi2013}
Z.~F\"uredi and M.~Simonovits.
\newblock The history of degenerate (bipartite) extremal graph problems.
\newblock In {\em Erd\"os centennial}, volume~25 of {\em Bolyai Soc. Math.
  Stud.}, pages 169--264. J\'anos Bolyai Math. Soc., Budapest, 2013.

\bibitem{Janzer2018}
O.~{Janzer}.
\newblock {Improved bounds for the extremal number of subdivision}.
\newblock {\em ArXiv e-prints}, 2018.

\bibitem{2018jiang}
T.~{Jiang}, J.~{Ma}, and L.~{Yepremyan}.
\newblock {On Tur\'{a}n exponents of bipartite graphs}.
\newblock {\em arXiv preprint arXiv:1806.02838}, June 2018.

\bibitem{kollar1996}
J.~Koll\'ar, L.~R\'onyai, and T.~Szab\'o.
\newblock Norm-graphs and bipartite {T}ur\'an numbers.
\newblock {\em Combinatorica}, 16(3):399--406, 1996.

\bibitem{kovari1954}
T.~K\"ovari, V.~T. S\'os, and P.~Tur\'an.
\newblock On a problem of {K}. {Z}arankiewicz.
\newblock {\em Colloquium Math.}, 3:50--57, 1954.

\bibitem{mantel1907}
W.~Mantel.
\newblock Problem 28.
\newblock {\em Wiskundige Opgaven}, 10(60-61):320, 1907.

\bibitem{verstraete2018}
J.~Verstra\"ete and J.~Williford.
\newblock Graphs without {T}heta subgraphs.
\newblock {\em Journal of Combinatorial Theory, Series B}, 2018.

\end{thebibliography}

\appendix
\section{Blow-up of a balanced rooted tree with radius three}

In this appendix, we present the proof of Theorem~\ref{thm-T47}. As $T_{4,7}$ is a balanced rooted tree, by Lemma~\ref{thm:lowerbound}, it suffices to prove the following proposition.
\begin{proposition}\label{prop:t_{4,7}}
	For each $\ell \in \mathbb{N}$, we have $\ex(n , T_{4,7}^\ell) = O(n^{10/7})$.
\end{proposition}

To prove Proposition~\ref{prop:t_{4,7}}, we need a variant of dependent random choice.

\begin{lemma}\label{lem:expand}
	Let $\ell \in \mathbb{N}$ and $G$ be a graph.
	Let $u_1 \in V(G)$, $A\subseteq \Gamma_{G}(u_1)$ and $B\subseteq \Gamma_G(A)\setminus \{u_1 \}$. If $e(A,B)\geq 4\ell |A|^2$ and $|B|\leq \frac{e(A,B) }{10\ell}$, then $G$ contains $T_{4,7}^{\ell}$ as a subgraph.
\end{lemma}
\begin{proof}
	Let $r_1,r_2,r_3,r_4$ be the four root vertices of $T_{4,7}^{\ell}$ in such a way that distance between $r_i$ and $r_{i+1}$ is three for each $i\in [3]$.
	For each $i\in [4]$, let $Z_i:=\Gamma_{T_{4,7}^{\ell}}(r_i)$.
	
	Let $B':= \{ b\in B: d_{G}(b,A) \geq 2\ell+2\}$.
	Then we have 
	$$ e(A,B') \geq e(A,B)- (2\ell+1)|B| \geq e(A,B)/2.$$
	
	We choose $u\in B'$ uniformly at random and let $X:=\Gamma_{G}(u,A)$.
	We say a pair $P\in \binom{A}{2}$ is \emph{bad} if $|N_{B'}(P)|\leq 2\ell$.
	Let $Y$ be the expected number of bad pairs in $X$. Then
	$$\mathbb{E}[Y] \leq \sum_{\text{$P$ a bad pair} } \mathbb{P}[ P\subseteq X]
	\leq \frac{2\ell}{|B'|} \binom{|A|}{2}\leq \frac{\ell |A|^2}{|B'|}.$$
	Let $X'$ be a subset of $X$ obtained by deleting an element from each bad pair in $X$. Since $|X'| \geq |X|-Y$ and $|B'|\leq |B|$, we have
	$$\mathbb{E}[|X'|] \geq \mathbb{E}[|X|] - \mathbb{E}[Y] \geq \frac{e(A,B')}{|B'|} -  \frac{\ell |A|^2}{|B'|} \geq \frac{e(A,B)}{4|B'|} \geq 2\ell+2.$$
	Hence, there exist a vertex $u_3\in B'$ and a set $X'\subseteq \Gamma_{G}(u_3)$ with $|X'|\geq 2\ell+2$ such that every pair $P\in \binom{X'}{2}$ has at least $2\ell + 1$ common neighbors in $B'$.
	
	Now, we construct an embedding $\phi$ of $T_{4,7}^{\ell}$ into $G$. We arbitrarily choose two vertices $u_2, u_4, \in X'$ and a subset $U_3$ of $X'\setminus \{u_2,u_4\}$ with $|U_3|=\ell$.
	Let $\phi(r_i)= u_i$ for each $i\in [4]$ and assign $\phi$ in an arbitrary way that $\phi(Z_3) = U_3\subseteq \Gamma_{G}(u_3)$. Note that $\phi$ is injective as $|Z_3|=|U_3|$ and $u_2,u_4\notin U_3$.
	
	For each $i\in \{2,4\}$ and each $x\in Z_i$, let $z_x \in Z_3$ be the unique neighbor of $x$ in $Z_3$. As $\phi(z_x)\in U_3 \subseteq X'$, we have $d_{G}(\{u_i, \phi(z_x)\},B ) \geq 2\ell+1$, we can define $\phi(x)$ in such a way that $\phi(x) \in N_{G}(\{u_1,\phi(z_x)\})\setminus\{u_3\}$ and $\phi$ is still injective. This is possible since the number of neighbours of $r_2$ or $r_4$ is $2\ell$.
	
	For each $x\in Z_1$, let $z'_x$ be the unique neighbour of $x$ in $Z_2$.
	We choose $\phi(x)$ from $\Gamma_{G}(\phi(z'_x), A) \setminus (U_3\cup \{u_2,u_4\})$ in such a way that $\phi$ is injective on $Z_1$. Since $|Z_1|=\ell$, it is possible by the definition of $B'$ as we have $\phi(z'_x)\in B'$. Since every vertex in $A$ is adjacent to $u_1 =\phi(r_1)$, this $\phi$ embeds a copy of $T_{4,7}^{\ell}$ into $G$.
\end{proof}

\begin{proof}[Proof of Proposition~\ref{prop:t_{4,7}}]
	Consider the numbers $n,q$ such that 
	$$ n \gg  q \gg \ell.$$
	Let $d:= n^{3/7}$. As before, it suffices to prove that an $n$-vertex graph $G$ with $\delta(G)\geq qd$ contains $T_{4,7}^{\ell}$ as a subgraph.
	The following claim will be useful for us.
	
	\begin{claim}\label{cl: clclaim}
		Suppose that we are given a vertex $u_0\in V(G)$ and subsets $A^{\#}\subseteq \Gamma_{G}(u_0)$ with $|A^{\#}|\leq \ell d$
		and $C\subseteq V(G)\setminus A^{\#}$ with $|C|\leq 10\ell d$.
		Then, there exist a vertex $u\in V(G)\setminus (A^{\#}\cup C \cup \{u_0\})$, sets $A\subseteq A^{\#}$, $B\subseteq \Gamma_{G}(u) \setminus (A^{\#} \cup C\cup \{u_0\})$ and a bijective function $f:B\rightarrow A$ satisfying the following. 
		For each $a\in A$, we have $f(a)a \in E(G)$ and $|A|=|B|\geq n^{-1/7} |A^{\#}|$.
	\end{claim}
	\begin{proof}
		Let $B^{\#} := \Gamma_G(A^{\#}) \setminus  (C\cup \{ u_0  \})$. 
		As $\delta(G)\geq qd$, for each $v\in V(G)$, we have 
		\begin{align}\label{eq: mindegr}
		|\Gamma_{G}(v)\setminus (A^{\#}\cup C\cup \{u_0\}) | \geq qd-10\ell d-|A^{\#}|-1 \geq qd/2.
		\end{align}
		
		For each $S\subseteq A^{\#}$, let $B_S:=\Gamma_G(S, B^{\#}) $.
		We claim that for each $S\subseteq A^{\#}$, we have
		\begin{align}\label{eq: qdS Hall}
		|B_S | \geq d |S|.
		\end{align}
		To show this, assume that we have a non-empty set $S\subseteq A^{\#}$ with $|B_S | < d |S|$.
		Since $|S| \leq |A^{\#}| \leq \ell d$ and $q\gg \ell$, by \eqref{eq: mindegr}, we have
		$$e(S,B_S) \geq \sum_{v\in S} d_{G}(v,B^{\#}) \geq qd|S|/2 \geq 4\ell |S|^2,$$
		and $|B_S| < d|S| \leq \frac{qd|S|}{20\ell} \leq \frac{e(S,B_S)}{10 \ell}$.
		Hence, we can apply Lemma~\ref{lem:expand} to $G$ with $u_0, S, B_S$ and $\ell$ playing the roles of $u_1, A, B$ and $\ell$ respectively to obtain a copy of $T_{4,7}^{\ell}$ in $G$, a contradiction. Hence \eqref{eq: qdS Hall} holds for all non-empty subset $S$ of $A^{\#}$.
		
		Thus  Lemma~\ref{lem:star} implies that there exists a collection $\{ \Gamma_a \subseteq \Gamma_{G}(a,B^{\#}): a\in A^{\#}\}$  of pairwise disjoint sets such that $|\Gamma_a| = d$ for all $a\in A^{\#}$.
		
		For each $a\in A^{\#}$, let $U_a:= \Gamma_G(\Gamma_a) \setminus (A^{\#}\cup C\cup \{u_0\})$.
		We claim that for each $a\in A^{\#}$, we have
		\begin{align}\label{eq: Ua size}
		|U_a| \geq d |\Gamma_a|.
		\end{align}
		Suppose there exists a vertex $a\in A^{\#}$ with $|U_a| < d |\Gamma_a|$.
		By \eqref{eq: mindegr}, for each $v\in \Gamma_a$, we have $d_{G}(v,U_a)\geq qd/2$, hence $e_G(\Gamma_a, U_a) \geq qd |\Gamma_a| /2 \geq 4\ell |\Gamma_a|^2$. Moreover, we have
		$$|U_a| < d |\Gamma_a| \leq \frac{qd|\Gamma_a|}{20\ell} \leq \frac{e_G(\Gamma_a, U_a) }{10 \ell}.$$
		Hence, we can apply Lemma~\ref{lem:expand} to $G$ with $a, \Gamma_a, U_a,$ and $\ell$ playing the roles of $u_1, A, B$, and $\ell$, respectively to obtain a copy of $T_{4,7}^{\ell}$ in $G$, a contradiction. Thus we obtain \eqref{eq: Ua size}.
		
		Let $U:= \bigcup_{a\in A^{\#} }U_a$ and consider an auxiliary biparitte graph $H$ with a vertex partition $(A^{\#},U)$ with 
		$$E(H)=\{ aw \in A^{\#}\times U : w\in U_a \}.$$
		For each $a \in A^{\#}$, \eqref{eq: Ua size} implies that $d_H(a) = |U_a| \geq d|\Gamma_a| = d^2$.
		Thus, by averaging, there exists a vertex $u\in U$ with 
		$$ d_{H}(u)\geq \frac{|E(H)|}{|U|} \geq \frac{ d^2 |A^{\#}| }{ n} \geq n^{-1/7}|A^{\#}|.$$
		Let $A:= \Gamma_{H}(u)\subseteq A^{\#}$. For each $a\in A$, 
		choose a vertex $f(a)\in \Gamma_a\cap \Gamma_{G}(u)$ which is a non-empty set by the definition of $H$ and choice of $A$. As $\{\Gamma_a:a \in A^{\#}\}$ is a collection of pairwise disjoint sets, the function $f$ is injective. Let $B:=f(A)$.
		From the construction, it is obvious that $A,B,C,\{u\}$ are pairwise disjoint.  Hence the set $A,B$ and a function $f$ are as desired. This proves the claim.
	\end{proof}

	Let us choose an arbitrary vertex $u_0 \in V(G)$ and a subset $A_0 \subseteq \Gamma_{G}(u_0)$ with $|A_0| = \ell d$.
	Let $B_0 := A_0$ and $C_0 := \emptyset$. 
	
	For each $1 \leq i \leq 3$ in order, let us apply Claim~\ref{cl: clclaim} with $u_{i-1}, B_{i-1}$ and $C_{i-1}$ where
	$$C_{i-1} := \bigcup_{j=0}^{i-2} (B_{j} \cup \{u_j \}) $$
	playing the roles of $u_0, A^{\#}$ and $C$ to obtain a vertex $u_i$, disjoint sets $A_i\subseteq B_{i-1}, B_i\subseteq \Gamma_{G}(u_i)\setminus C_{i-1}$ with $|A_i|=|B_i| \geq n^{-1/7} |B_{i-1}|$.
	and a bijective function $f_i: A_i\rightarrow B_i$ satisfying the following
	$$B_i\subseteq \Gamma_{G}(u_i) \text{ and }  a f_i(a) \in E(G) \text{ for each }a\in A_i.$$
	These repetitive applications of Claim~\ref{cl: clclaim} are possible, since we have $|C_i| \leq  3|A_0| \leq 3\ell d$.

Let $B'_3:=B_3$ and for each $i=2,1,0$ in order, we let
$B'_i:= f^{-1}_{i+1}(B'_{i+1})$. 
As $B'_i\subseteq A_i = B_{i-1}$, it follows that, for each $i\in [3]$,
	\begin{itemize}
	\item  there exists a perfect matching between $B'_{i-1}$ and $B'_i$ via $f_i$,
	\item $|B'_0|=\dots=|B'_3| \geq n^{-3/7}|A_0|$.
	\end{itemize}

	Now the sets $\{u_0\},\dots, \{u_3\}, B'_0,\dots, B'_3$ are all pairwise disjoint, and we claim that they induce a copy of $T_{4,7}^{\ell}$.	
	Indeed, $|B'_0|=\dots=|B'_3|\geq n^{-3/7}|A_0| = \ell$ and
	for each $j\in [3]$, there is a perfect matching between $B'_{j-1}$ and $B'_{j}$ as well as $B'_{i}\subseteq B_{i}\subseteq \Gamma_{G}(u_{i})$.	Hence, we obtain a copy of $T_{4,7}^{\ell}$, completing the proof.
\end{proof}

\end{document}